\documentclass[11pt,reqno,twoside]{amsart}
\usepackage[T1]{fontenc}
\usepackage[utf8]{inputenc}
\usepackage{lmodern,microtype}
\usepackage[a4paper,margin=30mm]{geometry}
\usepackage{amsmath,amssymb,amsthm,mathtools,mathrsfs,bm}
\usepackage{comment}
\usepackage{xcolor,enumitem,array}
\definecolor{cellblue}{RGB}{231,240,252}
\usepackage{tikz}
\usetikzlibrary{arrows.meta}
\usepackage[colorlinks=true,linkcolor=blue!55!black,citecolor=blue!55!black,urlcolor=blue!55!black]{hyperref}
\hypersetup{
pdftitle={The Shubin--Vakilian--Wolff Uncertainty Principle at Half Density},
pdfauthor={Ming Wang and Yunlei Wang}
}
\allowdisplaybreaks

  \numberwithin{equation}{section}
 \newtheorem{theorem}{Theorem}[section]
 \newtheorem{proposition}[theorem]{Proposition}
 \newtheorem{lemma}[theorem]{Lemma}
 \newtheorem{corollary}[theorem]{Corollary}
 
 \newtheorem*{svwtheorem}{SVW theorem}
 \newtheorem*{svwquestion}{SVW density question}
 \theoremstyle{definition}
 
\theoremstyle{remark}
\newtheorem{remark}[theorem]{Remark}

 \def\R {\mathbb{R}}
 \def\C {\mathbb{C}}
 \def\N {\mathbb{N}}
\def\Z {\mathbb{Z}}

 \def\T {\mathbb{T}}
 \def\one {\mathbf{1}}
 \def\cF {\mathcal{F}}

 \def\Ran {\operatorname{Ran}}

 \def\spec {\operatorname{spec}}
 \def\norm#1{\left\lVert #1\right\rVert}
 
 \def\ip#1#2{\left\langle #1,#2\right\rangle}
 \def\d {\,\mathrm{d}}
 \def\eps {\varepsilon}
 \def\e {\mathrm{e}}
 \def\i {\mathrm{i}}

\def\1{\mathbf{1}}
\def\cW {\mathcal{W}}
\def\supp {\operatorname{supp}}
\def\red {\operatorname{red}}
\def\AW {\operatorname{Op}^{\mathrm{AW}}}

\def\Re {\operatorname{Re}}
\def\Im {\operatorname{Im}}

\def\Ho {\mathcal{H}_{\mathrm{o}}}

 \title[Shubin--Vakilian--Wolff Uncertainty Principle]
   {The Shubin--Vakilian--Wolff Uncertainty Principle\\ at Half Density}
 \author[Ming Wang]{Ming Wang}
 \address[Ming Wang]{School of Mathematics and Statistics, Central South University, Changsha, Hunan 410083, P.R. China}
 \email{m.wang@csu.edu.cn}

\author[Yunlei Wang]{Yunlei Wang}
\address[Yunlei Wang]{Department of Mathematics, Louisiana State University, Baton Rouge, LA 70803, USA}
\email{ywang30@lsu.edu}
\subjclass[2020]{42A38, 47A30,  46L54}
\keywords{Fourier uncertainty principle, Wolff thin sets, critical density,
quadratic Fourier multipliers, odd Gaussian functions, anti-Wick operators}

\begin{document}
\begin{abstract}
Shubin, Vakilian, and Wolff proved a Fourier uncertainty principle for
sets of sufficiently small local density at the reciprocal scale and asked whether every density below one is admissible. We answer this question negatively in dimension one by constructing sequences of pairs of $1/2$-density thin sets and unit
vectors whose total position and Fourier mass outside these sets tends to zero. This obstruction
persists for every scaled reciprocal profile $\rho_\kappa(x)=\min\{1,\kappa/|x|\}$, $\kappa>0$. On the other hand, for $0<\kappa\le1$, the uncertainty estimate holds whenever the density
$$
\eps<\frac{1}{2(1+32\kappa)}.
$$
Thus the critical density tends to $1/2$ as $\kappa\downarrow0$. The obstruction
uses odd Gaussian packets to transfer norm bounds from a free-group model. The positive estimate uses a Fourier-complementary
 anti-Wick operator and quadratic straightening of the reciprocal geometry.
\end{abstract}
\maketitle
\setcounter{tocdepth}{2}
\tableofcontents

\section{Introduction and main results}\label{intro}

 \subsection{The Shubin--Vakilian--Wolff uncertainty principle}

A central quantitative question in Fourier uncertainty is whether a function and its Fourier transform
 can both be essentially concentrated on geometrically sparse sets \cite{HavinJoricke94}. Classical finite-measure forms include the Amrein--Berthier strong-annihilating-pair
 inequality, Benedicks' qualitative support theorem, Nazarov's quantitative estimate, and Jaming's higher-dimensional extension
\cite{AmreinBerthier77,Benedicks85,Jaming07,Nazarov93}. At a fixed observation scale, the Logvinenko--Sereda theorem instead
controls band-limited functions from thick observation sets \cite{LogvinenkoSereda74}, with quantitative refinements due to Kovrijkine \cite{Kovrijkine01}. Shubin, Vakilian, and Wolff (SVW)
 introduced a different geometric regime, in which sparsity is measured locally at a position-dependent scale.

For $d\ge1$, we use the Fourier convention
\begin{equation}\label{Fourier}
\mathcal{F}_d f(\xi) = \int_{\R^d} \e^{-2\pi i x\cdot\xi}f(x)\d x .
\end{equation}
 When $d=1$, we write $\mathcal{F}=\mathcal{F}_1$ and $\widehat f=\mathcal{F}f$.

Let
$$
\rho(x)=\min\{1,|x|^{-1}\}, \qquad \rho(0)=1.
$$
A measurable set $E\subset\R^d$ is called \emph{$\eps$-Wolff thin}\footnote{The original paper \cite{SVW98} simply
uses the term $\eps$-thin, while Kovrijkine \cite{Kovrizhkin03} refers to the
corresponding uncertainty principle as Wolff's version.} if
$$
|E\cap B(x,\rho(x))| \le \eps |B(x,\rho(x))|, \qquad x\in\R^d.
$$
  Thus sparsity is tested at unit scale near the origin and at reciprocal scale in the tail.

  Motivated by Anderson--Bernoulli models, SVW proved the following theorem \cite{SVW98}.

\begin{svwtheorem}
For every $d\ge1$, there exist constants $\eps_{\mathrm{SVW}}(d)>0$ and $C_d>0$ such that, whenever $E,F\subset\R^d$ are $\eps_{\mathrm{SVW}}(d)$-Wolff thin,
\begin{equation}\label{SVW}
 \|f\|_{L^2(\R^d)} \le C_d\left( \|\mathbf{1}_{E^c}f\|_{L^2(\R^d)} + \|\mathbf{1}_{F^c}\mathcal{F}_df\|_{L^2(\R^d)} \right)
\end{equation}
for every $f\in L^2(\R^d)$.
 \end{svwtheorem}

In dimension one, the reciprocal decay rate of the testing radius is
itself sharp: SVW showed that \eqref{SVW} can fail for a positive
continuous testing radius $\widetilde\rho$ satisfying
$$
 |x|\widetilde\rho(x)\longrightarrow\infty \qquad\text{as }|x|\longrightarrow\infty .
$$
Thus $\widetilde\rho(x)$ is asymptotically larger than the reciprocal scale $|x|^{-1}$. This sharpness concerns the decay rate, not the admissible density: for the reciprocal radius $\rho$, the SVW theorem still assumes a sufficiently small
dimension-dependent density, that is, some value $\varepsilon_{\mathrm{VSW}}(d)$. SVW asked whether this smallness restriction could be removed \cite[p.~940]{SVW98}.

\begin{svwquestion}
For every $d\ge1$ and every $0<\eps<1$, does an estimate of the form \eqref{SVW} hold, with a constant $C_{d,\eps}<\infty$ depending only on $d$ and $\eps$, for every pair of $\eps$-Wolff thin sets $E,F\subset\R^d$?
\end{svwquestion}

 Several subsequent works have extended the geometry or the applications of the SVW theorem. Kovrijkine \cite{Kovrizhkin03} treated distinct position and frequency testing radii, including the complementary power scales $\min\{1,|x|^{-a}\}$ and $\min\{1,|\xi|^{-1/a}\}$ for $a>0$. For continuous nonincreasing radial profiles, he identified a compatibility condition that is necessary and sufficient for an uncertainty principle at sufficiently small density.

Ghobber and Jaming established a Fourier--Bessel analogue of the SVW theorem \cite{GhobberJaming11}. They subsequently showed, for both the Fourier--Bessel and Euclidean Fourier transforms, that bounded cores may be adjoined to sufficiently
thin tails, with the admissible thinness depending on the core sizes \cite[Corollary~3.5 and Remark~3.6]{GhobberJaming13}. Samuelsen \cite{Samuelsen25} proved compactness of Fourier
concentration operators when the local Wolff densities of both sets tend to zero at infinity. More recently, Le~Balc'h and Yu \cite{LeBalcYu26} used the SVW theorem and Kovrizhkin's extension to obtain observability inequalities for the
Schr\"odinger equation from sets thick with respect to decaying densities. These results do not determine the largest fixed density for which the SVW estimate holds uniformly.

\subsection{Main results}

To study the dependence on the reciprocal scale, let
$$
 \rho_\kappa(x)=\min\{1,\kappa/|x|\},\qquad \rho_\kappa(0)=1,\qquad \kappa>0,
    $$
and write $I_x^{(\kappa)}=(x-\rho_\kappa(x),x+\rho_\kappa(x))$. A measurable set $E\subset\R$ is \emph{$(\eps,\kappa)$-Wolff thin} if
$$
|E\cap I_x^{(\kappa)}|\le2\eps\rho_\kappa(x), \qquad x\in\R.
$$
The case $\kappa=1$ is the original SVW geometry.

For measurable $E,F\subset\R$, let
$$
P_E=\mathbf{1}_E, \qquad Q_F=\mathcal{F}^{-1}P_F\mathcal{F}.
$$
The geometry of two orthogonal projections shows that an uncertainty estimate of the form \eqref{SVW} holds with a finite constant if and only if
$$
\|P_EQ_F\|_{2\to2}<1.
$$

For $0<\eps_1,\eps_2<1$, define
\begin{equation}\label{theta-k}
\Theta_{\kappa}(\eps_1,\eps_2) = \sup_{\substack{ E\text{ is }(\eps_1,\kappa)\text{-Wolff thin}\\
F\text{ is }(\eps_2,\kappa)\text{-Wolff thin} }} \|P_EQ_F\|_{2\to2},
\end{equation}
and write
$$
\Theta_{\kappa}(\eps) = \Theta_{\kappa}(\eps,\eps).
$$
The corresponding \emph{critical density} is
\begin{equation}\label{equ-intro-3}
   \eps_c(\kappa) := \sup\{\eps\in(0,1):\Theta_\kappa(\eps)<1\}.
\end{equation}
Since $\Theta_\kappa$ is nondecreasing in each density parameter, $\eps_c(\kappa)$ is the endpoint of the uniform uncertainty range.

For $a,b\in[0,1]$, define
\begin{equation}\label{equ-intro-Phi}
\Phi(a,b) =
\begin{cases}
\sqrt{a(1-b)}+\sqrt{b(1-a)}, & a+b\le1,\\
1,  & a+b\ge1.
\end{cases}
\end{equation}

Our first theorem gives a negative answer to the original SVW question and describes the obstruction side of the half-density transition.

\begin{theorem}\label{thm-main}
Let $\kappa>0$. For every $0<\eps_1,\eps_2<1$, we have
\begin{equation}\label{equ-main-1}
\Theta_{\kappa}(\eps_1,\eps_2) \ge \Phi(\eps_1,\eps_2).
\end{equation}
Consequently,
\begin{equation}\label{trace-line}
\Theta_{\kappa}(\eps_1,\eps_2)=1 \qquad \text{whenever }\eps_1+\eps_2\ge1,
\end{equation}
and hence
$$
\eps_c(\kappa)\le\frac12.
   $$
\end{theorem}

At $\kappa=1$ and $\eps_1=\eps_2=1/2$, the construction gives symmetric $1/2$-Wolff thin sets $E_n,F_n$ and odd unit functions $f_n\in L^2(\R)$ such that
$$
\|\one_{E_n^c}f_n\|_2+ \|\one_{F_n^c}\widehat f_n\|_2\longrightarrow0.
$$
Thus \eqref{SVW} fails at every density
$\eps\ge1/2$. 
For $\eps_1+\eps_2<1$, the theorem gives a lower bound for $\Theta_\kappa$, rather than its exact value. 

The next theorem provides a sufficient density condition such that the estimate \eqref{SVW} holds.

    \begin{theorem}[Half-density limit]\label{lowerbound9-10}
    For every $0<\kappa\le1$, we have
    \begin{equation}\label{equ-intro-5}
    \frac1{2(1+32\kappa)} \le \eps_c(\kappa) \le \frac12.
\end{equation}
    Consequently,
    \begin{equation}\label{equ-intro-5.5}
    \lim_{\kappa\downarrow0}\eps_c(\kappa) = \frac12.
    \end{equation}
    \end{theorem}

Thus every fixed density $\eps<1/2$ lies in the positive uncertainty range for all sufficiently small $\kappa$. Together with Theorem~\ref{thm-main}, this shows that the positive range approaches the same half-density threshold at which the obstruction persists. At the original scale, \eqref{equ-intro-5} gives the explicit bounds
$$
\frac1{66} \le \eps_c(1) \le \frac12.
$$
Thus, in dimension one the original SVW estimate holds uniformly throughout $0<\eps<1/66$:
$$
\|f\|_2 \le C\bigl( \|\mathbf 1_{E^c}f\|_2 + \|\mathbf 1_{F^c}\widehat f\|_2 \bigr),
$$
for all $\eps$-Wolff thin measurable sets $E,F\subset\R$, with $C$ universal. The upper bound is not claimed to be optimal and whether $\eps_c(1)=1/2$ remains open.

For simplicity, this article focuses on dimension $d=1$.
The half-density obstruction and the limit
$\lim_{\kappa\downarrow0}\eps_c(d,\kappa)=\frac12$
hold for every $d\ge1$ and have been formally verified in
Lean~4 using mathlib, see \cite{SVWLean}.
For the obstruction, our odd Gaussian packets give an explicit
and direct analytic realization of the free-group model underlying
the formalization.
For the positive estimate, both proofs use the same anti-Wick
construction: the formalization uses variable-radius averaging
to obtain $\eps_c(d,\kappa)\ge\frac12-C_d\kappa^{1/3}$ for
$0<\kappa\le1/8$, while our quadratic straightening and sharper
kernel estimates yield the one-dimensional bound
$\eps_c(\kappa)\ge\frac{1}{2(1+32\kappa)}$ for $0<\kappa\le1$.

\subsection{Ideas of the proof}

The obstruction and the positive result are proved by rather different mechanisms. On the obstruction side, the main difficulty is to construct sets that are sparse in the Wolff geometry but still support almost complete concentration in both position and frequency. On the positive side, one needs a mechanism that converts local density control into a uniform angle between the position and frequency subspaces.

 \medskip
 \noindent
 \textit{Outline of the proof of Theorem~\ref{thm-main}.} The proof is based on a quadratic family whose geometry is simple enough to control, while its position--frequency interaction is structured enough to admit sharp lower bounds.

\begin{enumerate}

\item[(1)] \textit{Quadratic periodic sets.} For $m\ge2$ and measurable $A\subset\T=\R/\Z$, define
$$
E_{m,A} := \left\{ x\in\R: \frac{mx^2}{2}\pmod 1\in A \right\}.
$$
At every fixed reciprocal scale $\kappa>0$, the Wolff density of $E_{m,A}$ converges to the phase density $|A|$ as $m\to\infty$. Thus, up to an asymptotically negligible error, prescribing the Wolff density reduces to prescribing the
measure of a subset of the torus.

\item[(2)] \textit{Noncommuting quadratic spectral projections.} Introduce the quadratic multiplier
$$
U_m f(x)=e^{\pi i m x^2}f(x)
$$
and its Fourier conjugate
  $$
V_m=\mathcal F^{-1}U_m\mathcal F.
$$
By functional calculus,
$$
   P_{E_{m,A}}=\mathbf 1_A(U_m), \qquad Q_{E_{m,B}}=\mathbf 1_B(V_m).
$$
 For open arcs $A,B\subset\T$, the main task is to prove
\[
\bigl\|P_{E_{m,A}}Q_{E_{m,B}}\bigr\|_{2\to2}
\ge \Phi(|A|,|B|).
\]
Both operators preserve the odd subspace
$$
\mathcal  H_{\mathrm o} = \{f\in L^2(\R):f(-x)=-f(x)\},
$$
and it is enough for our purpose to obtain the desired lower bound there. 

\item[(3)] \textit{Reduction to the free-group model.} The norm identity in Proposition~\ref{prp:norm-comparison}, proved in Appendix~\ref{apendix-prf-faithullness}, reduces estimates for continuous spectral cutoffs of $U_m,V_m$ on $\Ho$ to the corresponding estimates in the free-group model.

\item[(4)] \textit{The norm of two free projections.} In the free-group model, the spectral projections
$p=\one_A(\mathbf X)$ and $q=\one_B(\mathbf Y)$
are freely independent, with traces
$$
\alpha=|A|, \qquad \beta=|B|.
$$
The standard spectral law for two free projections gives
$$
\|pq\|_{2\to2}=\Phi(\alpha,\beta).
$$
In particular, the model exhibits the sharp transition $\alpha+\beta=1$.

  \item[(5)] \textit{Returning to the quadratic Fourier problem.} To pass from the auxiliary model back to the original spectral
  projections, we insert continuous cutoffs between slightly smaller and slightly larger arcs. The norm identity from step~(3), together with the projection calculation from step~(4), then yields
$$
  \bigl\| P_{E_{m,A}}Q_{E_{m,B}} \big|_{\mathcal H_{\mathrm o}} \bigr\|_{2\to2} \ge \Phi(|A_0|,|B_0|)
  $$
  for compact subarcs $A_0\Subset A$ and $B_0\Subset B$. Letting the compact subarcs exhaust $A$ and $B$, and using
the continuity of $\Phi$, gives the lower bound stated
in step~(2).

 Finally, choose the outer arc lengths slightly below $\eps_1,\eps_2$, take $m$ sufficiently large so that the quadratic sets satisfy the required Wolff-density constraints, and then let the
 inner and outer arc lengths approach the prescribed densities. This gives
 $$
 \Theta_\kappa(\eps_1,\eps_2) \ge \Phi(\eps_1,\eps_2),
$$
 and proves Theorem~\ref{thm-main}.
 \end{enumerate}

\medskip
 \noindent
\textit{Proof of Theorem~\ref{lowerbound9-10}: main steps.} The positive result uses a quite different mechanism. The main idea is to construct a positive auxiliary operator whose Fourier conjugate is its complement, and then use the Wolff density condition to separate quantitatively the position and frequency subspaces.

\begin{enumerate}

 \item[(1)] \textit{A positive anti-Wick separator.} We construct a positive anti-Wick contraction $R$ with nonnegative kernel
 and the exact Fourier-complementarity relation
 \begin{equation}
 \mathcal F R\mathcal F^{-1}=I-R. \label{eq}
 \end{equation}
 Thus the same positive operator can be used on both sides of the position--frequency problem: one seeks
 to show that $R$ is small on functions localized to a Wolff-thin spatial set, while $I-R$ is small on
 functions whose Fourier transforms are localized to a Wolff-thin frequency set. The positivity of the kernel makes these estimates accessible directly from geometric density information.

\item[(2)] \textit{Straightening the Wolff geometry and applying Schur's test.} On
 each half-line, the quadratic change of variables
$$
t=\pi x^2
$$
 turns the reciprocal spatial scale in the Wolff condition into intervals of the fixed length $4\pi\kappa$ in the $t$-variable. The Wolff-thinness assumption is therefore converted into a uniform weighted density bound at one fixed scale.

This change of variables is adapted to the quadratic structure of the kernel of $R$. Applying the resulting fixed-scale density estimate to its kernel rows gives uniform row and column bounds on
Wolff-thin sets. Schur's test then yields quantitative estimates for
$$
P_E R P_E \qquad\text{and}\qquad Q_F(I-R)Q_F.
$$

    \item[(3)] \textit{Energy separation and Fourier uncertainty.} By \eqref{eq}, the two compression estimates from step~(2) give complementary $R$-energy bounds on the position and frequency subspaces. A positive-contraction angle lemma then yields
$$
 \|P_EQ_F\|_{2\to2}<1.
  $$
  Keeping track of the constants gives
  $$
  \eps_c(\kappa) \ge \frac{1}{2(1+32\kappa)}, \qquad 0<\kappa\le1.
  $$
  Together with $\eps_c(\kappa)\le1/2$, this implies
  $$
  \lim_{\kappa\downarrow0}\eps_c(\kappa)=\frac12.
 $$
  \end{enumerate}

\medskip

We use the inner product
$$
  \ip{\phi}{\psi} = \int_{\R}\phi(x)\,\overline{\psi(x)}\d x
 $$
on $L^2(\R)$, and write $\norm{\phi}_2$ for the corresponding $L^2(\R)$ norm. Unless otherwise indicated, operator norms are taken on the Hilbert spaces on which the operators act.

  \section{Quadratic sets and the free model}
  \label{sec:q-set-norm-comparison}

 \subsection{Quadratic sets and Wolff density}
 \label{quadratic-sets}

  Fix $\kappa>0$. For a measurable set $E\subset\R$, define its \emph{$\kappa$-Wolff density} by
  \begin{equation}\label{wolff-density}
\delta_{\mathrm{W},\kappa}(E) := \sup_{x\in\R} \frac{|E\cap I_x^{(\kappa)}|}{|I_x^{(\kappa)}|} = \sup_{x\in\R} \frac{|E\cap I_x^{(\kappa)}|}{2\rho_\kappa(x)}.
  \end{equation}
   Thus $0\le\delta_{\mathrm{W},\kappa}(E)\le1$, and $E$ is $(\eps,\kappa)$-Wolff thin if and only if
$$
  \delta_{\mathrm{W},\kappa}(E)\le\eps.
$$
In particular, $\delta_{\mathrm{W},\kappa}(E)$ is the least density for which $E$ is $(\eps,\kappa)$-Wolff thin.

The Wolff density also controls the upper asymptotic density on the positive half-line.

\begin{lemma}\label{lem-density}
For every measurable $E\subset\R$,
\begin{equation}\label{equ-density-1}
\limsup_{R\to\infty} \frac{|E\cap[0,R]|}{R} \le \delta_{\mathrm{W},\kappa}(E).
\end{equation}
\end{lemma}

\begin{proof}
Put $\delta=\delta_{\mathrm{W},\kappa}(E)$ and
$$
x_n^{(\kappa)} = \sqrt\kappa\,(\sqrt n+\sqrt{n+1}), \qquad n\ge0.
 $$
Choose $N_\kappa$ so large that $x_n^{(\kappa)}\ge\kappa$ for any $n\ge N_\kappa$. Then
$$
I_{x_n^{(\kappa)}}^{(\kappa)} = (2\sqrt{\kappa n},2\sqrt{\kappa(n+1)}),
$$
    so these intervals tile the tail $(2\sqrt{\kappa N_\kappa},\infty)$.
Hence, for every integer $N\ge N_\kappa$,
$$
|E\cap(2\sqrt{\kappa N_\kappa},2\sqrt{\kappa N})| \le 2\delta\sqrt\kappa\, (\sqrt N-\sqrt{N_\kappa}).
$$
Now let $R\ge 2\sqrt{\kappa N_\kappa}$, and choose the unique integer $N\ge N_\kappa$ such that
$$
2\sqrt{\kappa N}\le R<2\sqrt{\kappa(N+1)},
$$
then, with $C_\kappa=|E\cap[0,2\sqrt{\kappa N_\kappa}]|$,
$$
\frac{|E\cap[0,R]|}{R}  \le \frac{C_\kappa}{R}   + \delta\left( \sqrt{\frac{N+1}{N}} - \sqrt{\frac{N_\kappa}{N}} \right).
$$
Letting $R$, and hence $N$, tend to infinity proves \eqref{equ-density-1}.
\end{proof}

For a measurable set $A\subset\T=\R/\Z$ and an integer $m\ge2$, define
$$
E_{m,A} = \left\{ x\in\R: \frac{mx^2}{2}\pmod 1\in A \right\}.
$$
 The following estimate is the only $\kappa$-dependent input needed for the obstruction results below.

\begin{proposition}	\label{prop-disc}
Let $A\subset\T$ be measurable, let $\alpha=|A|$, and let $m\ge2$. Then
\begin{equation}\label{disc-kappa}
\sup_{x\in\R} \left| \frac{|E_{m,A}\cap I_x^{(\kappa)}|} {|I_x^{(\kappa)}|} -\alpha \right| \le \Delta_{m,\kappa}, \quad\text{with }\, \Delta_{m,\kappa} :=    \frac4{\sqrt{m\min\{\kappa,1\}}}.
   \end{equation}
\end{proposition}

 \begin{proof}
Extend $g(t)=\one_A(t)-\alpha$ periodically to $\R$, and put
$$
    H(t)=\int_0^t g(s)\d s.
$$
Then $H$ is one-periodic and $\norm  H_{L^\infty(\R)}\le1$. Define
$$
J(R)=\int_0^R g(u^2)\d u.
$$
The integrand is even, and hence $J$ is odd. For $0\le R\le1$, the bound $|J(R)|\le 1$ is immediate. If $R\ge1$, integration by parts gives
$$
\int_1^R g(u^2)\d u = \left[\frac{H(u^2)}{2u}\right]_1^R + \int_1^R\frac{H(u^2)}{2u^2}\d u,
$$
    so
  $$
|J(R)|\le2,\qquad R\in\R.
 $$
With $c=m/2$, every interval $(a,b)$ therefore satisfies
\begin{equation}\label{equ-disc-3}
\left| \int_a^b g(cy^2)\d y \right| \le \frac4{\sqrt c}.
\end{equation}

By symmetry it is enough to take $x\ge0$. If $x\le\kappa$, then $|I_x^{(\kappa)}|=2$, and \eqref{equ-disc-3} gives
$$
\left| \frac{|E_{m,A}\cap I_x^{(\kappa)}|} {|I_x^{(\kappa)}|} -\alpha \right| =
\frac12 \left| \int_{I_x^{(\kappa)}} g(cy^2)\d y \right| \le \frac12\frac4{\sqrt c} = 2\sqrt{\frac2m} \le \Delta_{m,\kappa}.
$$
If $\kappa<x\le\sqrt{2\kappa}$, a range which is empty when $\kappa\ge2$, then $|I_x^{(\kappa)}|=2\kappa/x$, and
$$
\left| \frac{|E_{m,A}\cap I_x^{(\kappa)}|} {|I_x^{(\kappa)}|} -\alpha \right| \le \frac{2x}{\kappa\sqrt c} \le   \frac4{\sqrt{m\kappa}} \le \Delta_{m,\kappa}.
$$

It remains to consider $x>\kappa$ and $x\ge\sqrt{2\kappa}$. Put
 $$
a=x-\frac\kappa x, \qquad b=x+\frac\kappa x.
$$
Then $a>0$ and $|I_x^{(\kappa)}|=2\kappa/x$. Since $g(cy^2) = \frac1{2cy}\frac{\mathrm{d}}{\mathrm{d}y}H(cy^2),$ integration by parts gives
$$
\left| \int_a^b g(cy^2)\d y \right| \le \frac1{ca}.
$$
Hence
$$
\left| \frac{|E_{m,A}\cap I_x^{(\kappa)}|} {|I_x^{(\kappa)}|} -\alpha \right| \le \frac{x^2}{2\kappa c(x^2-\kappa)} \le \frac2{m\kappa}.
 $$
   If $0<\kappa\le1$ and $m\kappa\ge1$, this is at most $2/\sqrt{m\kappa}$; if $m\kappa<1$, the desired estimate is trivial. If $\kappa\ge1$, then
$$
\frac2{m\kappa} \le \frac4{\sqrt m} = \Delta_{m,\kappa}.
$$
This proves \eqref{disc-kappa}.
\end{proof}

The discrepancy estimate determines the Wolff density up to the same error.

  \begin{corollary}\label{cor-quadratic}
For every measurable $A\subset\T$ and every $m\ge2$,
\begin{equation}\label{quadratic-thin}
|A| \le \delta_{\mathrm{W},\kappa}(E_{m,A}) \le  |A|+\frac4{\sqrt{m\min\{\kappa,1\}}}.
\end{equation}
\end{corollary}

\begin{proof}
The upper bound follows immediately from \eqref{disc-kappa}. For the converse, the same function $J$ used above gives
$$
|E_{m,A}\cap[0,R]|-\alpha R = c^{-1/2}J(\sqrt c\,R),
$$
  and hence
$$
\lim_{R\to\infty} \frac{|E_{m,A}\cap[0,R]|}{R} = \alpha.
$$
Lemma~\ref{lem-density} therefore gives $\alpha\le\delta_{\mathrm{W},\kappa}(E_{m,A})$.
 \end{proof}

\subsection{Faithful reduced representation}\label{sec:norm-comparison}

We define the quadratic unitaries
\begin{equation}\label{eq:def-quadratic-unitaries}
    U_m=\e^{\pi \i mX^2} \quad\text{and}\quad V_m=\mathcal{F}^{-1}U_m\mathcal{F}=\e^{\pi \i m D^2},
\end{equation}
where $Xf(x)=xf(x)$ and $D=(2\pi \i)^{-1}\partial_x$.  By functional calculus, we define
\begin{equation*}
    P_{E_{m,A}}=\one_A(U_m) \quad\text{and}\quad Q_{E_{m,B}}=\one_B(V_m).
\end{equation*}
They preserve the odd subspace
\begin{equation*}
\Ho=\{f\in L^2(\R):f(-x)=-f(x)\}.
\end{equation*}
It is therefore sufficient to establish the required norm lower bounds on $\Ho$.

Let $\cW$ denote the set of \emph{reduced words} in $X^{\pm 1},Y^{\pm 1}$, with the empty word denoted by $o$. Write $\red(w)$ for the word obtained by cancelling adjacent inverse pairs. Let $(\delta_w)_{w\in\cW}$ be the standard orthonormal basis of $\ell^2(\cW)$ and put $\Omega=\delta_o$. Define 
\begin{equation*}
    \mathbf{X} \delta_w=\delta_{\red(Xw)} \quad \text{and}\quad \mathbf{Y} \delta_w =\delta_{\red(Yw)}.
\end{equation*}
The operators permute the orthonormal basis, hence they are unitary. In the following, norms of operators formed from $\mathbf{X}$ and $\mathbf{Y}$ are taken on $\ell^2(\cW)$.

Regarding $\cW$ as the free group under concatenation followed by reduction, let
\begin{equation*}
    C_r^*(\cW)=C^*(\mathbf{X},\mathbf{Y})\subset \mathcal{B} \bigl(\ell^2(\cW)\bigr)
\end{equation*}
be the operator-norm closure of the noncommutative Laurent polynomials in $\mathbf{X}$ and $\mathbf{Y}$.

\begin{proposition}\label{prp:norm-comparison}
    For every integer $m\ge 2$, the map on generators
    \begin{equation*}
        \mathbf{X}\longmapsto U_m|_{\Ho},\quad \mathbf{Y}\longmapsto V_m|_{\Ho}
    \end{equation*}
    extend to a faithful unital representation $\pi_m:C_r^*\to \mathcal{B}(\Ho)$. In particular, for every $\varphi,\psi\in C(\T)$,
    \begin{equation}\label{norm-comparison-main}
        \norm{\varphi(\mathbf{X})\psi(\mathbf{Y})}=\norm{\varphi(U_m)\psi(V_m)|_{\Ho}}.
    \end{equation}
\end{proposition}
By faithfulness of the representation $\pi_m$, we mean that $\pi_m$ is injective.
An explicit constructive proof is given in Appendix~\ref{apendix-prf-faithullness}. 

\begin{remark}
    To prove Theorem~\ref{thm-main}, only one-sided norm inequality
    \begin{equation*}
        \norm{\varphi(\mathbf{X})\psi(\mathbf{Y})}\le \norm{\varphi(U_m)\psi(V_m)|_{\Ho}}
    \end{equation*}
    is needed. We prove the full norm identity in Appendix~\ref{apendix-prf-faithullness} for its independent interest.
\end{remark}

\section{Free projections and the half-density obstruction}
\label{free-projections}

In this section, we give the norm formula for the lower bound and then finish the proof of Theorem~\ref{thm-main}. 
\subsection{Free projections and their norm formula}
We first present the standard spectral law for two free projections to the representation given by Section~\ref{sec:q-set-norm-comparison}.

We regard $\T=\R/\Z$ as the unit circle through the identification
$$
t\longmapsto \e^{2\pi i t}.
$$
Accordingly, if $U$ is unitary and $A\subset\T$ is measurable, then $\one_A(U)$ denotes the \emph{spectral projection} of $U$ associated with $\{\e^{2\pi i t}:t\in A\}$, and $|A|$ denotes normalized Haar measure on $\T$.

  Let $B,C\subset\T$ be measurable sets with
  \begin{equation}\label{stBC}
s=|B|\in(0,1), \qquad t=|C|\in(0,1),
  \end{equation}
  and set
  \begin{equation}\label{equ-pq}
  p=\one_B(\mathbf{X}), \qquad q=\one_C(\mathbf{Y}).
  \end{equation}

    \begin{lemma}\label{lem-root}
   For every reduced word $g$, the scalar spectral measures of $\mathbf{X}$ and $\mathbf{Y}$ associated with the basis vector $\delta_g$ are normalized Haar measure on $\T$. In particular, every $g\in \cW$ and all measurable sets $B,C\subset \T$,
  \begin{equation}\label{root-mass}
    \ip{\one_B(\mathbf{X})\delta_g}{\delta_g}=|B|, \qquad \ip{\one_C(\mathbf{Y})\delta_g}{\delta_g}=|C|.
  \end{equation}
  \end{lemma}

 \begin{proof}
 Let $g$ be a reduced word, and let $\mu_{\mathbf X,g}$ be the scalar spectral measure of $\mathbf X$ associated with $\delta_g$. By the spectral theorem, we have
 $$
 \int_{\T}\e^{2\pi\i kt}\,\mathrm d\mu_{\mathbf X,g}(t) =\ip{\mathbf X^k\delta_g}{\delta_g} =\begin{cases}1,&k=0,\\0,&k\ne0,\end{cases}
 $$
where the last equality follows from $\red(X^kg)\ne g$ for $k\ne0$. These are the Fourier coefficients of normalized Haar measure. So $\mu_{\mathbf{X},\delta_g}$ and the normalized Haar measure $\d t$ give exactly the same integrals against every integer-frequency exponential. By uniqueness, $\mu_{\mathbf X,g}$ is normalized Haar measure, which gives~\eqref{root-mass}.
 \end{proof}

Before we proceed, let us introduce von Neumann algebra and some related notions.
Let $\mathcal B(H)$ be the bounded linear operators on a complex Hilbert space $H$. A \emph{von Neumann algebra} on $H$ is a subset $\mathcal N\subseteq\mathcal B(H)$ containing the identity $I$, closed under complex linear combinations, products, and adjoints, and closed in the weak operator topology. The algebraic conditions mean that
$$
aA+bB,\quad AB,\quad A^*\in\mathcal N
\quad\forall A,B\in\mathcal N,\text{ and } \forall a,b\in\mathbb C.
$$

Let $\mathcal M=L(\cW)$ be the von Neumann algebra generated by $\mathbf X$ and $\mathbf Y$ on $\ell^2(\cW)$, equipped with its canonical faithful normal trace
$$
\tau(A)=\ip{A\Omega}{\Omega}.
$$

For a unitary operator $U$, we write $W^*(U)$ for the smallest von Neumann algebra containing $I$ and $U$. Since $U^*=U^{-1}$, we have
$$
W^*(U)=\overline{\operatorname{span}\{U^k:k\in\Z\}},
$$
where the closure is taken under the weak operator topology.

Two unital $*$-subalgebras $\mathcal{A}_1,\mathcal{A}_2\subseteq\mathcal{M}$ are called \emph{freely independent with respect to $\tau$} if, for every $n\ge1$, every choice of indices $i_1,\ldots,i_n\in\{1,2\}$ with $i_j\ne i_{j+1}$ for $1\le j<n$, and every choice of operators $A_j\in\mathcal A_{i_j}$ satisfying $\tau(A_j)=0$, we have
$$
\tau(A_1A_2\cdots A_n)=0.
$$

Define two von Neumann subalgebras $\mathcal A_1=W^*(\mathbf X)$ and $\mathcal A_2=W^*(\mathbf Y)$ of $\mathcal M$. The following lemma gives the free independence of them with respect to $\tau$.

\begin{lemma}\label{lem-free-independence}
    Two subalgebras $\mathcal{A}_1$ and $\mathcal{A}_2$ are freely independent with respect to $\tau$.
\end{lemma}

\begin{proof}
Put $T_1=\mathbf X$ and $T_2=\mathbf Y$.
Let $A_j\in\mathcal A_{i_j}$ satisfy $\tau(A_j)=0$,
where $i_j\ne i_{j+1}$ for $1\le j<n$.
By the spectral theorem, write
\[
A_j=h_j(T_{i_j}), \qquad h_j\in L^\infty(\T).
\]
Lemma~\ref{lem-root} gives $\int_\T h_j(t)\,\mathrm dt=0$.

Let $h_{j,N}$ be the Fej\'er means of $h_j$. They have
zero constant coefficient and satisfy
\[
h_{j,N}\longrightarrow h_j \quad\text{in }L^2(\T),
\qquad
\|h_{j,N}\|_\infty\le\|h_j\|_\infty.
\]
By Lemma~\ref{lem-root},
\[
\bigl\|(h_{j,N}(T_{i_j})-A_j)\delta_g\bigr\|^2
=\|h_{j,N}-h_j\|_{L^2(\T)}^2
\longrightarrow0
\qquad(g\in\cW).
\]
Uniform boundedness therefore gives strong convergence
$h_{j,N}(T_{i_j})\to A_j$.

Since the indices alternate and the constant coefficients
vanish, each product
\[
h_{1,N}(T_{i_1})\cdots h_{n,N}(T_{i_n})
\]
is a finite linear combination of nonempty reduced words,
and hence has zero trace. Passing to the strong limit yields
$\tau(A_1\cdots A_n)=0$, proving free independence.
\end{proof}

The spectral projections $p$ and $q$ belong to $\mathcal{A}_1$ and $\mathcal{A}_2$, respectively. By Lemmas~\ref{lem-root} and~\ref{lem-free-independence}, they are free projections with $\tau(p)=s$ and $\tau(q)=t$. The following proposition gives the formula for the norm of the product $\norm{pq}$.
\begin{proposition}\label{prop-tree}
Let $s,t,p,q$ be as in \eqref{stBC}--\eqref{equ-pq}, and let $\Phi$ be defined by \eqref{equ-intro-Phi}. Then
$$
\norm{pq}=\Phi(s,t).
$$
\end{proposition}

\begin{proof}
Let $\mu$ be the spectral distribution of $pqp$ associated with $\Omega$, so that $\mu(J)=\tau(\one_J(pqp))$ for Borel sets $J\subset\R$. The standard free-projection law \cite[Section~3.2, p.~321]{Collins05} states that $\mu$ consists of an absolutely continuous part with support $[x_-,x_+]$, where
\begin{equation}\label{edge-formula}
x_\pm=\bigl(\sqrt{s(1-t)}\pm\sqrt{t(1-s)}\bigr)^2,
\end{equation}
and possible atoms at $0$ and $1$, with
$$
\mu(\{1\})=\max\{s+t-1,0\}.
$$
The canonical trace $\tau$ on $\mathcal M=L(\cW)$ is faithful
in the sense of \cite[Definition~1.3]{NicaSpeicher06}, that is,
\[
\tau(A^*A)=0 \quad\Longrightarrow\quad A=0,
\qquad A\in\mathcal M.
\]
Indeed, for each $g\in\cW$, define the right translation
\[
\mathbf R_g\delta_w=\delta_{\red(wg)}, \qquad w\in\cW.
\]
Left and right translations commute, so $\mathbf R_g$ commutes
with $\mathbf X$ and $\mathbf Y$, and hence with every
$A\in\mathcal M$. If $\tau(A^*A)=0$, then
\[
0=\tau(A^*A)=\langle A^*A\Omega,\Omega\rangle
=\|A\Omega\|^2,
\]
so $A\Omega=0$. Since $\delta_g=\mathbf R_g\Omega$, we obtain
\[
A\delta_g=A\mathbf R_g\Omega
=\mathbf R_gA\Omega=0
\quad \forall g\in\cW.
\]
Thus $A$ vanishes on an orthonormal basis and therefore $A=0$,
which proves faithfulness. By
\cite[Proposition~3.15]{NicaSpeicher06}, it follows that
$\supp\mu=\spec(pqp)$. Consequently,
\begin{equation}\label{norm-reduction}
\norm{pq}^2=\norm{pqp}
=\begin{cases}
x_+, & s+t\le1,\\
1, & s+t>1.
\end{cases}
\end{equation}
Taking square roots gives the claimed formula. In particular, $x_+=1$ when $s+t=1$, so the norm reaches one at the half-density threshold.
\end{proof}

For completeness, Appendix~\ref{app:free-projections} gives a self-contained derivation of this standard norm formula using alternating words and a Jacobi operator. The relation between that argument and the cited free-projection law is explained at the beginning of the appendix.

\subsection{Transfer to Wolff-thin sets}
\label{proof-main}

We now combine the free-group norm formula with continuous cutoffs to prove Theorem~\ref{thm-main}.


The following lemma inserts compact subarcs and
 continuous cutoffs to transfer the free-model norm formula as a lower bound for the physical concentration operators.

\begin{lemma}\label{lem-cutoff}
Let $m\ge2$. Let $A,C\subset\T$ be open arcs and let
    $$
  B\Subset A, \qquad D\Subset C
  $$
  be compact subarcs. Then
  \begin{equation}\label{cutoff-transfer}
  \norm{ \one_A(U_m)\one_C(V_m)|_{\Ho} } \ge \Phi(|B|,|D|).
  \end{equation}
  \end{lemma}

\begin{proof}
Choose $\varphi,\psi\in C(\T)$ with $0\le\varphi,\psi\le1$ such that
$$
\varphi=1\ \text{on }B, \quad \supp\varphi\subset A, \qquad \psi=1\ \text{on }D,  \quad \supp\psi\subset C.
$$
In the free model set
$$
p=\one_B(\mathbf X), \qquad q=\one_D(\mathbf Y).
$$
For every unit vector $\zeta\in\Ran q$,
$$
    \psi(\mathbf Y)\zeta=\zeta, \qquad p\varphi(\mathbf X)=p,
$$
and hence
$$
\norm{\varphi(\mathbf X)\psi(\mathbf Y)} \ge \norm{pq} = \Phi(|B|,|D|)
$$
by Proposition~\ref{prop-tree}.  Proposition~\ref{prp:norm-comparison} (or Proposition~\ref{transfer-cont}) therefore gives
$$
 \norm{\varphi(U_m)\psi(V_m)|_{\Ho}} \ge  \Phi(|B|,|D|).
$$
Writing
$$
P=\one_A(U_m), \qquad Q=\one_C(V_m),
$$
the support conditions imply
$$
\varphi(U_m)\psi(V_m) = \varphi(U_m)PQ\psi(V_m).
$$
Since the two functional-calculus factors are contractions and preserve $\Ho$,
$$
\norm{\varphi(U_m)\psi(V_m)|_{\Ho}} \le \norm{PQ|_{\Ho}},
$$
which proves the claim.
\end{proof}

Now we finish the proof of Theorem~\ref{thm-main}.

\begin{proof}[Proof of Theorem~\ref{thm-main}]
Fix $0<\eps_1,\eps_2<1$. Choose open arcs $A_j,C_j$ and compact subarcs
$$
B_j\Subset A_j, \qquad D_j\Subset C_j
$$
such that
$$
|A_j|<\eps_1, \qquad |C_j|<\eps_2,
$$
and
$$
|B_j|,|A_j|\longrightarrow\eps_1, \qquad |D_j|,|C_j|\longrightarrow\eps_2.
$$
Choose an integer $m_j\ge2$ so large that
\begin{equation}\label{scale-choice9-10}
\frac4{\sqrt{m_j\min\{\kappa,1\}}} \le \min\{ \eps_1-|A_j|, \eps_2-|C_j| \}.
\end{equation}
Set $E_j=E_{m_j,A_j},	F_j=E_{m_j,C_j}$. Thanks to \eqref{scale-choice9-10},
Corollary~\ref{cor-quadratic} shows that $E_j$ is $(\eps_1,\kappa)$-Wolff thin and $F_j$ is $(\eps_2,\kappa)$-Wolff thin. Since
$$
P_{E_j}=\one_{A_j}(U_{m_j}), \qquad Q_{F_j}=\one_{C_j}(V_{m_j}),
$$
Lemma~\ref{lem-cutoff} gives
$$
\Theta_\kappa(\eps_1,\eps_2) \ge \norm{P_{E_j}Q_{F_j}|_{\Ho}} \ge \Phi(|B_j|,|D_j|).
$$
Letting $j\to\infty$ yields
$$
\Theta_\kappa(\eps_1,\eps_2) \ge \Phi(\eps_1,\eps_2).
$$
If $\eps_1+\eps_2\ge1$, then $\Phi(\eps_1,\eps_2)=1$, while a product of two orthogonal projections has norm at most one. Thus
$$
\Theta_\kappa(\eps_1,\eps_2)=1.
$$
Taking $\eps_1=\eps_2=1/2$ gives $\eps_c(\kappa)\le1/2$.
 \end{proof}

\section{Positive bounds for arbitrary sets}
\label{positive-part}

Throughout this section, $0<\kappa\le1$. We prove the lower bound in Theorem~\ref{lowerbound9-10} using a positive operator constructed by standard anti-Wick quantization, see \cite[Sections~II--III]{Daubechies88} and \cite{Folland89,Grochenig01}. We estimate its kernel over Wolff-thin sets after a quadratic change of variables. Schur's test and positivity then yield the uncertainty estimate.

\subsection{The anti-Wick separator}
Let
$$
g(u)=2^{1/4}\e^{-\pi u^2},\qquad
 g_{x,\xi}(u)=\e^{2\pi i\xi(u-x/2)}g(u-x),
$$
and, for a bounded measurable symbol $a$, define $\AW(a) f$ for any $f\in \mathcal{S}(\R) $ by
$$
\left(\AW(a)f\right)(y)=\int_{\R^2}a(x,\xi)\ip{f}{g_{x,\xi}}g_{x,\xi}(y)\d x\d\xi
$$
and then extend to all $L^2$ functions. It is easy to check
\begin{equation*}
    \int_{\R^2}|\langle f,g_{x,\xi}\rangle|^2\d x\mathrm{d}\xi=\norm{f}_{L^2(\R)}^2,
\end{equation*}
and it gives $\AW(1)=I$. So $0\le\AW(a)\le I$ whenever $0\le a\le1$. For this construction, we refer to \cite{Daubechies88,Folland89,Grochenig01}. Take
$$
a(x,\xi)=\frac{x^2}{x^2+\xi^2}\quad ((x,\xi)\ne(0,0)),
\qquad a(0,0)=0,\qquad R_{\rm cell}=\AW(a).
$$
Since $\cF g_{x,\xi}=g_{\xi,-x}$ and $a(\xi,-x)=1-a(x,\xi)$ almost everywhere,
\begin{equation}\label{Fourier-complement}
0\le R_{\rm cell}\le I,\qquad
\cF R_{\rm cell}\cF^{-1}=I-R_{\rm cell}.
\end{equation}
Denote by $K_{\rm cell}$ the kernel of $R_{\rm cell}$.
Then the elementary Fourier transform
$$
\int_\R\e^{2\pi i\xi u}\frac{x^2}{x^2+\xi^2}\d\xi
=\pi|x|\e^{-2\pi|x||u|}\qquad(x\ne0)
$$
gives the nonnegative symmetric kernel
\begin{equation}\label{Kcell}
K_{\rm cell}(y,z)=\sqrt2\,\pi\int_\R |x|\,
\e^{-\pi[(y-x)^2+(z-x)^2]-2\pi|x||y-z|}\d x.
\end{equation}

\subsection{Quadratic straightening and row estimates}
 The purpose of this subsection is to convert Wolff thinness into a uniform row bound for the nonnegative
 kernel $K_{\rm cell}$. There are two main steps. First, the
 quadratic coordinate $t=\pi z^2$ turns the reciprocal Wolff intervals on each ray into intervals of one fixed length. Second, in this coordinate a row of $K_{\rm cell}$ splits into a unimodal principal profile and a decreasing remainder, both
 of which can be controlled by the resulting fixed-cell density condition.

Let $E$ be an $(\eps,\kappa)$-Wolff thin set, where $0<\eps\le1/2$, and put
\begin{align}\label{equ-low-1}
\lambda=2\pi\kappa.
\end{align}
 On the positive and negative rays, define
 \begin{align}\label{mu-eta}
 \mathrm{d}\mu(t)=\frac{\mathrm{d}t}{2\sqrt{\pi t}}, \qquad \eta_\pm(t) =
\mathbf{1}_E\!\left(\pm\sqrt{\frac t\pi}\right), \qquad   t>0.
 \end{align}

\begin{lemma}\label{straightening}
For every $a\ge0$,
\begin{equation}\label{ray-thin}
\int_a^{a+2\lambda}\eta_\pm\d\mu \le \eps\mu([a,a+2\lambda]).
\end{equation}
\end{lemma}

\begin{proof}
For every $a\ge0$ there is a unique $u\ge\sqrt\kappa$ such that
$$
a=\pi\left(u-\frac\kappa u\right)^2.
$$
 Since $0<\kappa\le1$ and $u\ge\sqrt\kappa$, $\rho_\kappa(u)=\frac{\kappa}{u}$, so
$$
I_u^{(\kappa)} = \left( u-\frac{\kappa}{u}, u+\frac{\kappa}{u} \right) \subset[0,\infty).
$$
Under $t=\pi z^2$, its image has length
$$
\pi\bigl(u+\rho_\kappa(u)\bigr)^2 - \pi\bigl(u-\rho_\kappa(u)\bigr)^2 = 4\pi u\rho_\kappa(u) = 4\pi\kappa =  2\lambda.
$$
Hence $I_u^{(\kappa)}$ is mapped onto $(a,a+2\lambda)$. By \eqref{mu-eta}
and the $(\eps,\kappa)$-Wolff thinness of $E$, we have
$$
\int_a^{a+2\lambda}\eta_+(t)\d\mu(t) = \int_{I_u^{(\kappa)}}\mathbf{1}_E(z)\d z = |E\cap I_u^{(\kappa)}| \le \eps |I_u^{(\kappa)}| = \eps\mu([a,a+2\lambda]).
$$
The estimate for $\eta_-$ follows by reflection.
\end{proof}

Thus, from this point on, the geometry of $E$ enters only through the fixed-cell estimate \eqref{ray-thin}. We next record the two consequences that will be used for the kernel rows.

\begin{lemma}\label{lem-profile}
Let $\lambda$ and $\mathrm{d}\mu$ be given by \eqref{equ-low-1}-\eqref{mu-eta}. Let $0\le\eta\le1$ satisfy
\begin{equation}\label{cell-assumption}
\int_a^{a+2\lambda}\eta\d\mu \le \eps\mu([a,a+2\lambda]) \qquad(a\ge0).
\end{equation}
Then the following hold.

\begin{enumerate}
\item[(i)] If $h$ is nonnegative, decreasing, and integrable with respect to $\mu$, then
\begin{equation}\label{eq-profile-1}
\int_0^\infty \eta h\d\mu \le \eps\int_0^\infty h\d\mu + \eps(1-\eps) \sqrt{\frac{2\lambda}{\pi}}\,h(0).
\end{equation}

\item[(ii)] For $T\ge0$, set
\begin{equation}\label{unimodal}
k_T(t) = \sqrt\pi\min\{\sqrt T,\sqrt t\}\,\e^{-|T-t|}, \qquad t\ge0.
\end{equation}
Then
\begin{equation}\label{eq-profile-3}
\int_0^\infty \eta k_T\d\mu \le \eps\int_0^\infty k_T\d\mu + \eps(1-\eps)
\begin{cases}
\sqrt{2T\lambda},&T\le\lambda,\\[1mm]
\sqrt2\,\lambda,&T\ge\lambda.
\end{cases}
\end{equation}
\end{enumerate}
   \end{lemma}

\begin{proof}
$(i)$ Partition $[0,\infty)$ into
$$
  J_j=[2j\lambda,2(j+1)\lambda],  \qquad j\ge0,
$$
and write
$$
h_j=h(2(j+1)\lambda), \qquad H_j=h(2j\lambda)-h_j.
$$
Since $h$ is decreasing, $0\le h-h_j\le H_j$ on $J_j$. Moreover, \eqref{cell-assumption} gives
$$
 \int_{J_j}(\eta-\eps)\d\mu\le0.
$$
Therefore
$$
\begin{aligned}
\int_{J_j}(\eta-\eps)h\d\mu \le \int_{J_j}(\eta-\eps)(h-h_j)\d\mu\le (1-\eps)H_j\int_{J_j}\eta\d\mu\le \eps(1-\eps)\mu(J_j)H_j.
\end{aligned}
$$
Since $\mu(J_j)\le\mu([0,2\lambda]) =\sqrt{\frac{2\lambda}{\pi}}$,	summing over $j$
and using the telescoping bound
  $$
 \sum_{j\ge0}H_j\le h(0)
$$
yields
  $$
\int_0^\infty(\eta-\eps)h\d\mu \le \eps(1-\eps) \sqrt{\frac{2\lambda}{\pi}}\,h(0).
$$
Rearranging gives \eqref{eq-profile-1}.

\noindent
$(ii)$ The case $T=0$ is trivial. Assume $T>0$ and set
$$
L=2\lambda, \qquad K_T=\max k_T=\sqrt{\pi T}.
$$
Extend $k_T$ by zero to $(-\infty,0)$.

\smallskip
\noindent
\emph{Step 1: Reduction to one cell.} Define
$$
D(u)   = \eps\mu([0,u]) - \int_0^u\eta\d\mu.
$$
By \eqref{cell-assumption},
$$
D(u+L)\ge D(u), \qquad D'(u)=\frac{\eps-\eta(u)}{2\sqrt{\pi u}}.
$$
Integration by parts gives
\begin{equation}\label{profile-ibp}
\int_0^\infty\eta k_T\d\mu  = \eps\int_0^\infty k_T\d\mu + \int_0^\infty k_T'(u)D(u)\d u,
\end{equation}
where the boundary terms vanish since $D(u)=O(\sqrt u)$ at the origin and $k_T$
decays exponentially at infinity.

Let $I=[a,a+L]$, where $a\ge0$ and $T\in I$, and set
\[
S(s)=\sum_{j\in\Z}k_T(s+jL), \qquad s\in I.
\]
For almost every $s\in I$ and every $j$ with $s+jL>0$,
the inequality $D(u+L)\ge D(u)$ gives
\[
D(s+jL)\le D(s) \quad\text{if }j<0,
\qquad
D(s+jL)\ge D(s) \quad\text{if }j>0.
\]
Since $T\in I$, we also have
$k_T'(s+jL)\ge0$ for $j<0$ and
$k_T'(s+jL)\le0$ for $j>0$. Therefore
\[
\sum_{\substack{j\in\Z\\s+jL>0}}
k_T'(s+jL)D(s+jL)
\le
D(s)\sum_{\substack{j\in\Z\\s+jL>0}}k_T'(s+jL)
=
D(s)S'(s),
\]
where the last equality uses the zero extension of $k_T$
to $(-\infty,0)$.
Integrating over $I$ and using the periodicity $S(a+L)=S(a)$, we obtain
\begin{equation}\label{equ-cell-01}
 \int_0^\infty\eta k_T\d\mu \le \eps\int_0^\infty  k_T\d\mu + \int_I(\eta-\eps)\bigl(S-S(a)\bigr)\d\mu.
\end{equation}
For an interval $J$, define the total variation of $h$ on $J$ by
\[
\operatorname{Var}_J(h)
:=
\sup_{\substack{n\ge1,\;t_0<\cdots<t_n\\
                 t_0,\ldots,t_n\in J}}
\sum_{j=1}^n |h(t_j)-h(t_{j-1})|.
\]
The zero extension of $k_T$ increases from $0$ to $K_T$ and then decreases back to $0$, so 
$$
\operatorname{Var}_{\R}(k_T)=2K_T.
 $$
 
    Since the intervals $I+jL$, $j\in\Z$, partition $\R$, the periodization satisfies
$$
\operatorname{Var}_{I}(S) \le \sum_{j\in\Z} \operatorname{Var}_{I+jL}(k_T) = 2K_T.
$$
As $S$ is $L$-periodic, a full period contains both a passage from $\min S$ to $\max S$ and a return from $\max S$ to $\min S$; hence $2\operatorname{osc}S \le   \operatorname{Var}_{I}(S)$. Therefore
    \begin{equation}\label{osc-period}
\operatorname{osc}S\le K_T.
\end{equation}

\smallskip
\noindent
\emph{Step 2: The ranges $T<\lambda$ and $T\ge3\lambda$.} Suppose first that $T<\lambda$. Since $S$ is decreasing on $(T,L)$ and
$S(L)=S(0)$, we may choose $a\in[0,T]$ such that $S(a)=\min S$. Thus
$$
0\le S-S(a)\le\operatorname{osc}S \qquad\text{on }I,
$$
and \eqref{cell-assumption} gives
$$
\begin{aligned}
  \int_I(\eta-\eps)\bigl(S-S(a)\bigr)\d\mu \le (1-\eps)\operatorname{osc}S \int_I\eta\d\mu\le \eps(1-\eps) \mu(I)\operatorname{osc}S.
\end{aligned}
$$
  Since $a\ge0$,
$$
   \mu(I)\operatorname{osc}S \le \mu([0,L])K_T = \sqrt{\frac{2\lambda}{\pi}}\,K_T = \sqrt{2T\lambda}.
$$
Together with \eqref{equ-cell-01}, this proves the first bound in \eqref{eq-profile-3}.

  Now suppose that $T\ge3\lambda$. Choose $a\in[T-L,T]$ with $S(a)=\min S$. The same argument gives
 $$
 \int_I(\eta-\eps)\bigl(S-S(a)\bigr)\d\mu \le \eps(1-\eps) \mu(I)\operatorname{osc}S.
$$
 Since $\mu$ has decreasing density,
 $$
 \begin{aligned}
 \mu(I)\operatorname{osc}S \le K_T\mu([T-2\lambda,T])= \frac{2\lambda} {1+\sqrt{1-2\lambda/T}} < \sqrt2\,\lambda.
 \end{aligned}
 $$
 This proves the second bound in this range.

\noindent
\emph{Step 3: The transition range $\lambda\le T\le3\lambda$.} Take the centered cell
$$
    I=[T-\lambda,T+\lambda], \qquad a=T-\lambda.
$$
We claim that
\begin{equation}\label{equ-phase-1}
S(a)-\min S<\frac{K_T}{2}.
\end{equation}
Since $S$ is $L$-periodic, it is enough to compare $S(a)$ with $S(T-r)$ for $0\le r\le L$. From the definition of $k_T$,
\begin{equation}\label{phase-profile}
   \frac{S(T-r)}{K_T} = \sum_{j\ge0} \e^{-(r+jL)} \sqrt{\left(1-\frac{r+jL}{T}\right)_+} + \frac{\e^{-(L-r)}}{1-\e^{-L}}.
\end{equation}

  For $0\le r\le\lambda$, the sum in \eqref{phase-profile} is decreasing in $r$, and therefore
  $$
  S(a)-S(T-r) \le \frac{K_T(\e^{-\lambda}-\e^{-2\lambda})} {1-\e^{-2\lambda}} = \frac{K_T}{1+\e^\lambda} < \frac{K_T}{2}.
    $$

It remains to consider $r=\lambda+q$, $0\le q\le\lambda$. Set
$$
\alpha=\frac{T-\lambda}{\lambda}\in[0,2], \qquad x=\frac q\lambda\in[0,1].
$$
Since $T\le3\lambda$, \eqref{phase-profile} yields
 \begin{equation}\label{phase-diff}
\frac{S(a)-S(T-\lambda-q)}{K_T} = \e^{-\lambda} \left[ \sqrt{\frac{\alpha}{1+\alpha}} - \e^{-q}\sqrt{\frac{(\alpha-x)_+}{1+\alpha}} - \frac{\e^q-1}{1-\e^{-2\lambda}} \right].
\end{equation}
If $\lambda\ge1/2$, the right-hand side is bounded by
  $$
\e^{-1/2}\sqrt{\frac23}<\frac12.
$$
For $0<\lambda<1/2$, using
$$
\e^{-q}\ge1-q, \qquad \e^q-1\ge q, \qquad \frac{\lambda}{1-\e^{-2\lambda}} \ge\frac{1+\lambda}{2},
$$
  we obtain from \eqref{phase-diff}
\begin{equation}\label{equ-phase-3}
\frac{S(a)-S(T-\lambda-q)}{K_T} \le \e^{-\lambda} F_\alpha(x), \qquad
    F_\alpha(x) = \sqrt{\frac{\alpha}{1+\alpha}} - \sqrt{\frac{(\alpha-x)_+}{1+\alpha}} - \frac{x}{3}.
\end{equation}
An elementary endpoint argument gives
\begin{equation}\label{equ-phase-3.5}
F_\alpha(x)<\frac12, \qquad 0\le\alpha\le2,\quad 0\le x\le1.
\end{equation}
Indeed, for $x\ge\alpha$ one uses
$$
F_\alpha(x) \le \sqrt{\frac{\alpha}{1+\alpha}}-\frac{\alpha}{3}<\frac12,
$$
   while for $x\le\alpha$ the function $x\mapsto F_\alpha(x)$ is convex, so it suffices to check the relevant endpoints
$x=0,\alpha,1$. This proves \eqref{equ-phase-1}.

We now return directly to the last term in \eqref{equ-cell-01}. Put
 $$
   \delta=S(a)-\min S.
 $$
 By \eqref{osc-period} and \eqref{equ-phase-1},
 $$
0\le S-\min S\le K_T, \qquad  0\le\delta<\frac{K_T}{2}\le(1-\eps)K_T.
 $$
 Since $S-S(a)=(S-\min S)-\delta$,	the cell condition \eqref{cell-assumption} gives
 \begin{align}
 \int_I(\eta-\eps)(S-S(a))\d\mu &\le (1-\eps)K_T\int_I\eta\d\mu + \delta\left( \eps\mu(I)-\int_I\eta\d\mu \right) \nonumber\\
 &\le \eps(1-\eps)K_T\mu(I). \label{cell-est}
 \end{align}
 Finally, writing $R=T/\lambda\ge1$,
 \begin{equation}\label{cell-mass}
   K_T\mu(I) = \lambda \frac{2\sqrt R} {\sqrt{R+1}+\sqrt{R-1}} \le \sqrt2\,\lambda.
 \end{equation}
 Combining \eqref{equ-cell-01}, \eqref{cell-est}, and \eqref{cell-mass} gives
 $$
  \int_0^\infty\eta k_T\d\mu \le \eps\int_0^\infty k_T\d\mu + \eps(1-\eps)\sqrt2\,\lambda,
 $$
as required.
 \end{proof}

  We can now estimate a complete row of $K_{\rm cell}$.

\begin{proposition}\label{row-estimate9-10}
For $y\in\R$, put $T=\pi y^2$ and
$$
m(T)  := \int_{\R}K_{\rm cell}(y,z)\d z.
$$
Then
\begin{equation}\label{row-mass}
\frac12\le m(T)<1.
\end{equation}
Moreover,
\begin{align}
\int_EK_{\rm cell}(y,z)\d z &< \eps m(T) + \eps(1-\eps) \bigl(2(1-m(T))+5\lambda m(T)\bigr),
  &&0<\lambda\le1, \label{equ-row-small}\\
\int_EK_{\rm cell}(y,z)\d z &< \eps m(T) + 3\lambda\eps(1-\eps), &&\lambda\ge1. \label{row-large}
\end{align}
\end{proposition}

    \begin{proof}
    By reflection it suffices to consider $y\ge0$. Set $T=\pi y^2$ and define
   $$
    b(u) = \frac{\e^{-\pi u^2}}{\sqrt2} \int_0^\infty v\e^{-v^2-2\sqrt{2\pi}\,uv}\d v.
$$
    A direct calculation from \eqref{Kcell} gives, for $y,z\ge0$,
    \begin{align}
    K_{\rm cell}(y,z) &= \pi\min\{y,z\}\e^{-\pi|y^2-z^2|} + \e^{-\pi y^2}b(z) + \e^{-\pi z^2}b(y), \label{K-same}\\
    K_{\rm cell}(y,-z) &= \e^{-\pi y^2}b(z) + \e^{-\pi z^2}b(y). \label{K-opposite}
    \end{align}

Under $t=\pi z^2$, the first term in \eqref{K-same} becomes the unimodal profile $k_T$ from \eqref{unimodal}, while the remaining terms form the decreasing profile
$$
h_T(t) = \e^{-T}b\!\left(\sqrt{\frac t\pi}\right) + \e^{-t}b\!\left(\sqrt{\frac T\pi}\right).
$$
Writing $c=h_T(0)$, Lemma~\ref{lem-profile}, applied on the two rays, yields
\begin{equation}\label{equ-row-4}
\int_EK_{\rm cell}(y,z)\d z \le \eps m(T) + \eps(1-\eps) \left( a_T   + 2\sqrt{\frac{2\lambda}{\pi}}\,c \right),
\end{equation}
where
$$
a_T \le
\begin{cases}
\sqrt{2T\lambda},&T\le\lambda,\\[1mm]
\sqrt2\,\lambda,&T\ge\lambda.
\end{cases}
$$
In either case,
\begin{equation}\label{row-error}
a_T + 2\sqrt{\frac{2\lambda}{\pi}}\,c \le \sqrt2\,\lambda + 2\sqrt{\frac2\pi}\,c\sqrt\lambda.
\end{equation}

 Direct integration of \eqref{K-same}--\eqref{K-opposite} gives
 \begin{equation}\label{mT}
 m(T) = \frac12 +  \frac12\int_0^T \frac{\e^{-r}}{\sqrt{1+r/T}}\d r,
\end{equation}
 with $m(0)=1/2$, proving \eqref{row-mass}. Put
 $$
 d=1-m(T).
 $$
 The same calculation gives
 $$
 d \ge \frac{\e^{-T}}2, \qquad c \le \frac{\e^{-T}}{\sqrt2},
 $$
 and hence
\begin{equation}\label{equ-row-6}
 \frac{c^2}{d}\le\e^{-T}\le1.
 \end{equation}

Suppose first that $0<\lambda\le1$. By AM--GM and \eqref{equ-row-6},
$$
2\sqrt{\frac2\pi}\,c\sqrt\lambda \le d+\frac{2c^2}{\pi d}\lambda \le d+\frac2\pi\lambda.
$$
Together with \eqref{row-error},
$$
a_T + 2\sqrt{\frac{2\lambda}{\pi}}\,c < d+\frac52\lambda < 2d+5\lambda m(T),
$$
where we used $m(T)\ge1/2$. Substitution into \eqref{equ-row-4} proves \eqref{equ-row-small}.

If $\lambda\ge1$, then $c\le1/\sqrt2$ and $\sqrt\lambda\le\lambda$, so \eqref{row-error} gives
   $$
 a_T + 2\sqrt{\frac{2\lambda}{\pi}}\,c \le \left(\sqrt2+\frac2{\sqrt\pi}\right)\lambda < 3\lambda.
$$
   Substituting again into \eqref{equ-row-4} proves \eqref{row-large}.
   \end{proof}

\subsection{Proof of Theorem \ref{lowerbound9-10}}

The row estimates above provide the two compression bounds used in the projection argument.

\begin{lemma}	\label{lem-angle}
Let $P,Q$ be orthogonal projections and suppose
$$
\norm{PQ}\le\theta<1.
$$
    Then
$$
\norm f \le \frac1{1-\theta} \left( \norm{(I-P)f} + \norm{(I-Q)f} \right).
$$
\end{lemma}

  \begin{proof}
  Since
  $$
  Pf=PQf+P(I-Q)f,
$$
  we have
  $$
  \norm{Pf} \le \theta\norm f+\norm{(I-Q)f}.
  $$
 Combine this with
  \begin{equation*}
  \norm f \le \norm{(I-P)f}+\norm{Pf}. \qedhere
  \end{equation*}
  \end{proof}

  \begin{lemma}\label{positive-angle}
Let $P,Q$ be orthogonal projections, let $0\le R\le I$, and let $a,b\in[0,1]$ satisfy $a+b\le1$. If
$$
PRP\le aP, \qquad Q(I-R)Q\le bQ,
$$
then
$$
\norm{QP} \le \sqrt{a(1-b)}+\sqrt{b(1-a)}.
  $$
In particular, if $a=b=s<1/2$, then
$$
   \norm{QP} \le 2\sqrt{s(1-s)} <1.
$$
\end{lemma}
\begin{proof}
Let $p\in\Ran P$ and $q\in\Ran Q$ be unit vectors, and write $u=\ip{Rp}{p}\le a$ and $v=\ip{(I-R)q}{q}\le  b$. By the identity $I=R+(I-R)$ and the Cauchy--Schwarz inequality, we have
$$
|\ip pq| \le \sqrt{u(1-v)}+\sqrt{v(1-u)} \le \sqrt{a(1-b)}+\sqrt{b(1-a)}.
$$
In the last step, we used the fact that the middle expression is increasing in each variable on
$u,v\ge0$, $u+v\le1$. Taking the supremum over all such $p,q$ proves the result.
\end{proof}

\begin{proof}[Proof of Theorem~\ref{lowerbound9-10}]
Fix $0<\kappa\le1$ and let $E,F$ be $(\eps_*,\kappa)$-Wolff thin, where
$$
   \lambda=2\pi\kappa, \qquad C_0=\frac{32}{\pi}, \qquad \eps_* = \frac1{2+C_0\lambda} = \frac1{2(1+32\kappa)}.
$$
 Set
 $$
\tau=\eps_*(1-\eps_*), \qquad \gamma(\lambda) = \frac{\frac12-\eps_*}{\tau}.
$$
A direct calculation gives
\begin{equation}\label{equ-constant-1}
\frac{\gamma(\lambda)}{\lambda} = \frac{C_0}{2} \left( 1+\frac1{1+C_0\lambda} \right) > \frac{16}{\pi} > 5.
\end{equation}

For $y\in\R$, put $T=\pi y^2$. By Proposition~\ref{row-estimate9-10},
\begin{equation}\label{m-range}
\frac12\le m(T)<1.
\end{equation}

  Suppose first that $0<\lambda\le1$. By \eqref{equ-row-small},
    $$
\int_EK_{\rm cell}(y,z)\d z < \Phi_\lambda(m(T)),
    $$
    where
$$
    \Phi_\lambda(m) = \eps_*m + \tau\bigl(2(1-m)+5\lambda m\bigr).
    $$
    Since $\Phi_\lambda$ is affine and $m(T)\in[1/2,1)$, it is enough to check the endpoints. At $m=1/2$,
    $$
    1-2\Phi_\lambda(1/2) = \frac{ \lambda(C_0-5)(1+C_0\lambda) }{ (2+C_0\lambda)^2 } >0,
$$
    while
$$
    \Phi_\lambda(1) = \eps_*+5\lambda\tau  < \frac12
    $$
 by \eqref{equ-constant-1}. Hence
    \begin{equation}\label{small-gap}
    \int_EK_{\rm cell}(y,z)\d z < A_<(\lambda) := \max\{\Phi_\lambda(1/2),\Phi_\lambda(1)\}  < \frac12.
    \end{equation}

 If $1\le\lambda\le2\pi$, then \eqref{row-large} and $m(T)<1$ give
\begin{equation}\label{equ-gap-large}
\int_EK_{\rm cell}(y,z)\d z < A_>(\lambda) := \eps_*+3\lambda\tau < \frac12,
\end{equation}
since $3\lambda<\gamma(\lambda)$ by \eqref{equ-constant-1}.

Define
$$
A(\kappa) =
\begin{cases}
 A_<(2\pi\kappa),&2\pi\kappa\le1,\\
A_>(2\pi\kappa),&2\pi\kappa >1.
 \end{cases}
$$
Then $A(\kappa)<1/2$, and \eqref{small-gap}--\eqref{equ-gap-large} hold uniformly in $y$. Since $K_{\rm cell}$ is nonnegative and symmetric, Schur's test yields
\begin{equation}\label{PRP}
P_ER_{\rm cell}P_E \le A(\kappa)P_E.
\end{equation}
Applying the same estimate to $F$ and using
$$
\mathcal F R_{\rm cell}\mathcal F^{-1} = I-R_{\rm cell},
$$
we obtain
\begin{equation}\label{QRQ}
Q_F(I-R_{\rm cell})Q_F \le A(\kappa)Q_F.
\end{equation}

 Lemma~\ref{positive-angle}, applied to \eqref{PRP}--\eqref{QRQ}, gives
$$
 \norm{P_EQ_F}_{2\to2} \le 2\sqrt{A(\kappa)(1-A(\kappa))} <1.
 $$
 Hence Lemma~\ref{lem-angle} gives, for every $f\in L^2(\R)$,
 $$
 \norm f_2 \le \frac{ \norm{\one_{E^c}f}_2 + \norm{\one_{F^c}\widehat f}_2 }{ 1-2\sqrt{A(\kappa)(1-A(\kappa))} }.
  $$
 Therefore
 $$
 \eps_c(\kappa) \ge \frac1{2(1+32\kappa)}, \qquad 0<\kappa\le1.
 $$

Together with Theorem~\ref{thm-main},
$$
\frac1{2(1+32\kappa)} \le \eps_c(\kappa) \le \frac12.
$$
In particular, $\eps_c(1)\ge\frac1{66}$, and letting $\kappa\downarrow0$ gives
\begin{equation*}
\lim_{\kappa\downarrow0}\eps_c(\kappa)=\frac12. \qedhere
\end{equation*}
\end{proof}

\appendix
\section{Proof of the faithful reduced representation}\label{apendix-prf-faithullness}

In this appendix, we give the proof of Proposition~\ref{prp:norm-comparison}. It is an explicit and direct analytic realization of the algebraic proof given in \cite{SVWLean}. 
Throughout this appendix, we fix an integer $m\ge2$ and use the notation of
Section~\ref{sec:q-set-norm-comparison}, with $U_m,V_m$ as in
\eqref{eq:def-quadratic-unitaries}.

\subsection{Odd Gaussian packets and M\"obius dynamics}

For $z\in\mathbb H$, define
    $$
    \psi_z(x)=2^{5/4}\sqrt\pi\,(\Im z)^{3/4}x\e^{\pi\i zx^2}.
    $$
    The following lemma gives the covariance and overlap formulas for these packets. For their metaplectic interpretation, we refer to \cite{Folland89}.

\begin{lemma}\label{lem-gaussian}
 Each $\psi_z$ belongs to $\Ho$ and has norm one. The subspace $\Ho$ is invariant under $U_m$ and $V_m$. For $z,w\in\mathbb{H}$,
 $$
\bigl|\ip{\psi_z}{\psi_w}\bigr| = \left( \frac{2\sqrt{\Im z\,\Im w}}{|z-\overline w|} \right)^{3/2}.
$$
Moreover, up to unimodular scalar factors, we have covariance relations
$$
 U_m\psi_z\sim\psi_{a_m(z)}, \qquad V_m\psi_z\sim\psi_{b_m(z)}, \qquad \cF\psi_z\sim\psi_{-1/z},
$$
where
$$
a_m(z)=z+m, \qquad b_m(z)=\frac{z}{1-mz}.
$$
  \end{lemma}

\begin{proof}
The multiplier $\e^{\pi\i mx^2}$ is even, so $U_m$ preserves parity. The Fourier transform preserves the even and odd subspaces, and hence so does $V_m=\cF^{-1}U_m\cF$.

The normalization and overlap formula follow from
$$
\int_{\R}x^2\e^{-ax^2}\d x = \frac{\sqrt\pi}{2a^{3/2}}, \qquad \Re a>0.
$$
Indeed,
 $$
|\psi_z(x)|^2 = 2^{5/2}\pi(\Im z)^{3/2}x^2 \e^{-2\pi(\Im z)x^2},
$$
so the preceding Gaussian integral gives $\|\psi_z\|_2=1$. For the overlap, combine
the exponential factors into
$$
\exp\!\left(\pi\i(z-\overline  w)x^2\right)
$$
and apply the same Gaussian integral with a complex parameter of positive real part. Taking absolute values gives the displayed formula.

The identity for $U_m$ is immediate, since multiplication by $\e^{\pi\i mx^2}$ replaces $z$ by $z+m$. Completing the square in the Fourier transform of $x\e^{\pi\i zx^2}$ gives
$$
\cF\psi_z\sim\psi_{-1/z}.
$$
The formula for $V_m$ then follows from $V_m=\cF^{-1}U_m\cF$.
\end{proof}

We use the standard M\"obius action
$$
T_M(z)=\frac{az+b}{cz+d}, \qquad
M=\begin{pmatrix}a&b\\c&d\end{pmatrix}\in SL_2(\R).
$$
The maps $a_m$ and $b_m$ are represented by
$$
A_m=\begin{pmatrix}1&m\\0&1\end{pmatrix}, \qquad
B_m=\begin{pmatrix}1&0\\-m&1\end{pmatrix}\in SL_2(\Z),
$$
respectively. Since $T_{MN}=T_M\circ T_N$, compositions and inverses correspond to matrix products and inverses, with the rightmost factor acting first.

\begin{lemma}\label{word-action}
Let $\gamma$ be a finite word in $a_m^{\pm1},b_m^{\pm1}$, and let $M_\gamma$ and $W_\gamma$ be the products obtained by substituting $(A_m,B_m)$ and $(U_m,V_m)$, respectively, in the same order. Write
\begin{equation}\label{equ-word-1}
M_\gamma=\begin{pmatrix}a&b\\c&d\end{pmatrix}\in SL_2(\Z).
\end{equation}
Then, for $z\in\mathbb H$,
\begin{align}
\gamma(z)&=\frac{az+b}{cz+d}, \label{mobius-form}\\
\Im\gamma(z)&=\frac{\Im z}{|cz+d|^2}. \label{Im-mobius}
\end{align}
In particular, $\gamma(\mathbb H)\subset\mathbb H$, and, up to a unimodular scalar,
\begin{equation}\label{equ-word-3}
W_\gamma\psi_z\sim\psi_{\gamma(z)}.
\end{equation}
\end{lemma}

\begin{proof}
It's easy to check that M\"obius action gives \eqref{mobius-form} and \eqref{Im-mobius} directly. The covariance identities in Lemma~\ref{lem-gaussian} also hold for the inverse generators. The finite step iteration gives \eqref{equ-word-3}.
\end{proof}

For $\xi\in\R$ with $c\xi+d\ne0$, the preceding formulas give
\begin{equation}\label{equ-boundary-01}
\gamma(\xi+it)\longrightarrow\gamma(\xi)\in\R, \qquad
\Im\gamma(\xi+it)=\frac{t}{|c\xi+d|^2}+o(t)
\quad\text{as } t\downarrow 0.
\end{equation}
Together with Lemma~\ref{lem-gaussian}, these asymptotics yield asymptotic orthogonality of packets whose limiting parameters are distinct.

  \begin{lemma}\label{lem-ping-pong}
For $m\ge2$, after cancelling adjacent inverse pairs, every nontrivial finite composition of $a_m^{\pm1}$ and $b_m^{\pm1}$ defines a nonidentity M\"obius transformation. Consequently, distinct such compositions define distinct M\"obius transformations.
\end{lemma}

\begin{proof}
In $\widehat{\C}=\C\cup\{\infty\}$, set
$$
\mathcal{X}=\{|z|<1\}, \qquad \mathcal{Y}=\{|z|>1\}\cup\{\infty\}.
$$
For every nonzero integer $k$,
$$
a_m^k(z)=z+km, \qquad b_m^k(z)=\frac{z}{1-kmz}.
 $$
If $z\in\mathcal{X}$, then $|a_m^k(z)| \ge |km|-|z| >1$, so
$$
a_m^k(\mathcal{X})\subset\mathcal{Y} \qquad(k\ne0).
$$
If $z\in\mathcal{Y}\setminus\{\infty\}$, then
$$
|1-kmz| \ge |kmz|-1 >|z|,
$$
and hence $|b_m^k(z)|<1$.	Moreover,
$$
b_m^k(\infty)=-(km)^{-1}\in\mathcal{X}.
$$
Therefore
$$
b_m^k(\mathcal{Y})\subset\mathcal{X} \qquad(k\ne0).
$$

After cancelling adjacent inverse pairs, every nontrivial finite composition can be written as alternating nonzero powers of $a_m$ and $b_m$. The preceding inclusions then imply, by the ping-pong lemma \cite[Lemma~II.24]{deLaHarpe00}, the maps $a_m$ and $b_m$ freely generate a subgroup of the M\"obius group. Hence the last assertion is proved.
\end{proof}

\subsection{The reduced-word realization}
For $w\in\cW$, let $\gamma_w=w(a_m,b_m)$. By Lemma~\ref{lem-ping-pong}, distinct reduced words give distinct M\"obius transformations. We next choose one real point at which all these transformations are finite and have pairwise distinct values.

 \begin{lemma}\label{generic-point}
 There exists $\xi\in\R$ such that
 \begin{equation}\label{a-generic}
 \xi\ \text{is not a pole of any }\gamma_w \quad  (w\in\cW),
\end{equation}
 and
 \begin{equation}\label{equ-a-2}
 \gamma_w(\xi)\ne\gamma_v(\xi) \qquad (w\ne v).
 \end{equation}
 \end{lemma}

   \begin{proof}
By Lemma~\ref{word-action}, every word map has a determinant-one representative
$$
\gamma_w(z) = \frac{a_wz+b_w}{c_wz+d_w}.
$$
Hence $\gamma_w$ has at most one finite real pole.

  Now fix two distinct reduced words $w\ne v$. By Lemma~\ref{lem-ping-pong}, the maps $\gamma_w$ and $\gamma_v$ are distinct. Away from their poles, a collision
$$
\gamma_w(\xi)=\gamma_v(\xi)
$$
implies
$$
(a_w\xi+b_w)(c_v\xi+d_v) - (a_v\xi+b_v)(c_w\xi+d_w) =0.
$$
The left-hand side is a polynomial of degree at most two, and it cannot vanish identically because the two fractional-linear maps are distinct. Thus each pair of distinct word maps has only
finitely many real collision points.

Since $\cW$ is countable, the union of all poles and pairwise collision points is countable. Any real number outside this exceptional set has the required properties.
\end{proof}

Fix from now on one such $\xi$, and set
$$
z_t=\xi+it, \qquad t\downarrow0.
$$
For $w\in\cW$, define the boundary packet
$$
 \Psi_{w,t} = w(U_m,V_m)\psi_{z_t} \in\Ho.
$$
The next lemma is the analytic input in the transfer argument: on every fixed finite
set of reduced words, these packets become orthonormal as the base point approaches the real axis.

\begin{lemma}\label{lem-gram}
For every finite family of distinct reduced words $w_1,\dots,w_N\in\cW$,
$$
    \left( \left\langle \Psi_{w_i,t}, \Psi_{w_j,t} \right\rangle \right)_{1\le i,j\le N} \longrightarrow I_N \qquad \text{as } t\downarrow 0.
 $$
More precisely, for $i\ne j$,
$$
\left| \left\langle \Psi_{w_i,t}, \Psi_{w_j,t} \right\rangle \right| = O(t^{3/2}).
$$
\end{lemma}

\begin{proof}
  Let $\gamma_j=\gamma_{w_j}$. Since $\xi$ is not a pole of any $\gamma_j$, it follows from \eqref{equ-boundary-01} that
\begin{equation}\label{equ-transfer-01}
\Im\gamma_j(z_t)=O(t),\qquad \gamma_j(z_t)\longrightarrow\gamma_j(\xi)\in\R.
\end{equation}
Moreover, the choice of $\xi$ gives, for $i\ne j$,
\begin{equation}\label{denom-limit}
\bigl|\gamma_i(z_t)-\overline{\gamma_j(z_t)}\bigr| \longrightarrow |\gamma_i(\xi)-\gamma_j(\xi)|>0.
 \end{equation}
 By the covariance and overlap formulas in Lemma~\ref{lem-gaussian}, we have
$$
\bigl|\langle\Psi_{w_i,t},\Psi_{w_j,t}\rangle\bigr| =\left( \frac{2\sqrt{\Im\gamma_i(z_t)\,\Im\gamma_j(z_t)}} {|\gamma_i(z_t)-\overline{\gamma_j(z_t)}|} \right)^{3/2}.
$$
Combining the preceding estimates, we obtain
$$
\bigl|\langle\Psi_{w_i,t},\Psi_{w_j,t}\rangle\bigr| =O(t^{3/2}),\qquad i\ne j.
$$
Since each $\Psi_{w_j,t}$ has norm one, the diagonal entries are one. This proves the convergence of the Gram matrices.
\end{proof}

For $t>0$ and any finitely supported vector
$$
c=\sum_{w\in\cW} c_w\delta_w\in\ell^2(\cW),
$$
define
$$
J_t c = \sum_{w\in\cW} c_w\,\Psi_{w,t} = \sum_{w\in\cW} c_w\,w(U_m,V_m)\psi_{z_t}.
$$
Thus $J_t$ is defined on the linear span of the standard basis
$(\delta_w)_{w\in\cW}$.

 We then give some basic properties of $J_t$. By Lemma~\ref{lem-gram}, we have
\begin{equation}
    \norm{J_t c}_2 \longrightarrow \norm{c}_{\ell^2(\cW)} \qquad \text{as } t\downarrow 0.\label{eq:preceding-norm-convergence}
\end{equation}
Moreover, by the definition of the packets, we obtain 
\begin{equation}\label{intertwining}
U_mJ_t c=J_t\mathbf{X}c, \qquad V_mJ_t c=J_t\mathbf{Y}c.
\end{equation}
For any finite Laurent polynomial $P$ in two noncommuting unitary variables, the linearity and iteration give
$$
P(U_m,V_m)J_t c = J_tP(\mathbf{X},\mathbf{Y})c.
$$

The next proposition transfers operator norm lower bounds from the reduced-word model to the quadratic operators.

\begin{proposition}\label{prop-transfer}
For any finitely supported $c\in\ell^2(\cW)$ and polynomial $P$ given above, 
$$
\norm{P(U_m,V_m)J_t c}_2 \longrightarrow \norm{P(\mathbf{X},\mathbf{Y})c}_{\ell^2(\cW)}.
$$
In particular,
 $$
\norm{P(\mathbf{X},\mathbf{Y})} \le \norm{P(U_m,V_m)|_{\Ho}}.
  $$
\end{proposition}
\begin{proof}
 Since $P(\mathbf X,\mathbf Y)c$ is also finitely supported, by Lemma~\ref{lem-gram} we have
$$
\|P(U_m,V_m)J_tc\|_2 \longrightarrow\|P(\mathbf X,\mathbf Y)c\|_{\ell^2(\cW)}.
$$
Combining this limit with \eqref{eq:preceding-norm-convergence} gives
$$
\|P(\mathbf X,\mathbf Y)c\|_{\ell^2(\cW)}  \le \|P(U_m,V_m)|_{\Ho}\|\, \|c\|_{\ell^2(\cW)}.
$$
Taking the supremum over finitely supported unit vectors proves the norm inequality.
  \end{proof}

    The polynomial transfer extends immediately to continuous functional calculus.

\begin{proposition}\label{transfer-cont}
For every $\varphi,\psi\in C(\T)$,
\begin{equation}\label{equ-transfer-9-10}
\norm{\varphi(\mathbf{X})\psi(\mathbf{Y})} \le \norm{\varphi(U_m)\psi(V_m)|_{\Ho}}.
\end{equation}
\end{proposition}

\begin{proof}
Choose Laurent polynomials $p_n,q_n$ converging uniformly to
$\varphi,\psi$, respectively.
 Continuous functional calculus
gives, for every unitary $T$,
\begin{equation}\label{fc-approx}
\norm{p_n(T)-\varphi(T)} \le \norm{p_n-\varphi}_\infty, \qquad \norm{q_n(T)-\psi(T)} \le \norm{q_n-\psi}_\infty.
\end{equation}
  Apply Proposition~\ref{prop-transfer} to
 $$
P_n(X,Y)=p_n(X)q_n(Y)
$$
to obtain
\begin{equation}\label{equ-approx-2}
\norm{p_n(\mathbf{X})q_n(\mathbf{Y})} \le \norm{p_n(U_m)q_n(V_m)|_{\Ho}}.
\end{equation}
The two products converge in operator norm by \eqref{fc-approx}. Passing to the limit in \eqref{equ-approx-2} proves \eqref{equ-transfer-9-10}.
\end{proof}

\subsection{The reverse norm inequality and faithfulness}
\label{app-reverse-norm}

We now prove the reverse inequality by estimating Gaussian matrix
coefficients. Write
$$
\rho_m(w)=w(U_m,V_m)|_{\Ho},\qquad
\lambda(w)=w(\mathbf X,\mathbf Y),\qquad w\in\cW.
$$
These are unitary representations of the free group $\cW$ generated by $X$ and $Y$.
We extend them linearly to the group algebra $\C[\cW]$, consisting of finite
complex linear combinations of words, with involution $w^*=w^{-1}$.
For each word $w$, set
$$
M_w=w(A_m,B_m)\in SL_2(\Z).
$$
Lemma~\ref{lem-ping-pong} implies that $w\mapsto M_w$ is injective.
We use the Frobenius norm
$$
\norm{M}_{\mathrm F}=(a^2+b^2+c^2+d^2)^{1/2},\qquad
M=\begin{pmatrix}a&b\\c&d\end{pmatrix}.
$$

\begin{lemma}\label{lem-integral-matrix-count}
There is a constant $C$ such that, for every $R\ge1$,
\begin{equation}\label{eq-integral-matrix-count}
\#\{M\in SL_2(\Z):\norm{M}_{\mathrm F}\le R\}\le CR^2.
\end{equation}
Consequently,
\begin{equation}\label{eq-matrix-summability}
\sum_{w\in\cW}\norm{M_w}_{\mathrm F}^{-3}<\infty.
\end{equation}
\end{lemma}

\begin{proof}
Let $u\in\Z^2$ be the first column of a matrix in $SL_2(\Z)$.
Then $u$ is nonzero and primitive. If $v_0$ is one possible second column,
then all possible second columns must be of the form $v_0+ku$, $k\in\Z$.
Among them, at most $1+2R/|u|$ have Euclidean norm, which we denote by $|\cdot|$, at most $R$.
Thus the left-hand side of \eqref{eq-integral-matrix-count} is bounded by
$$
\sum_{\substack{u\in\Z^2\\0<|u|\le R}}
\left(1+\frac{2R}{|u|}\right)\le CR^2.
$$
Here we used $\#\{u\in\Z^2:|u|\le R\}=O(R^2)$ and
$\sum_{0<|u|\le R}|u|^{-1}=O(R)$, where the latter follows by grouping the
points in dyadic annuli. 

Since $w\mapsto M_w$ is injective and
$\norm{M_w}_{\mathrm F}\ge\sqrt2$, another dyadic decomposition gives
\[
\sum_{w\in\cW}\norm{M_w}_{\mathrm F}^{-3}
\le\sum_{j=0}^\infty 2^{-3j}
\#\{M\in SL_2(\Z):\norm{M}_{\mathrm F}<2^{j+1}\}
\le C\sum_{j=0}^\infty 2^{-j}<\infty.          \qedhere
\]
\end{proof}

\begin{lemma}\label{lem-mobius-invariance}
For $z,\zeta\in\mathbb H$, define
$$
K(z,\zeta)=\frac{2\sqrt{\Im z\,\Im\zeta}}{|z-\overline\zeta|}.
$$
Then, for every $M\in SL_2(\R)$,
$$
K(T_M(z),T_M(\zeta))=K(z,\zeta).
$$
\end{lemma}

\begin{proof}
Write $M=\begin{pmatrix}a&b\\c&d\end{pmatrix}\in SL_2(\R)$. Since the coefficients are real,
$\overline{T_M(\zeta)}=(a\overline\zeta+b)/(c\overline\zeta+d)$.
Subtracting the two fractions gives
\begin{align*}
T_M(z)-\overline{T_M(\zeta)}
&=\frac{(az+b)(c\overline\zeta+d)
-(a\overline\zeta+b)(cz+d)}{(cz+d)(c\overline\zeta+d)}\\
&=\frac{(ad-bc)(z-\overline\zeta)}{(cz+d)(c\overline\zeta+d)}\\
&=\frac{z-\overline\zeta}{(cz+d)(c\overline\zeta+d)}.
\end{align*}
Taking $\zeta=z$ and dividing by $2\i$, we obtain
$$
\Im T_M(z)=\frac{\Im z}{|cz+d|^2}.
$$
The same formula holds with $z$ replaced by $\zeta$. Hence
$$
2\sqrt{\Im T_M(z)\,\Im T_M(\zeta)}
=\frac{2\sqrt{\Im z\,\Im\zeta}}{|cz+d|\,|c\zeta+d|}.
$$
On the other hand, taking absolute values in the difference identity
and using $|c\overline\zeta+d|=|c\zeta+d|$ gives
$$
|T_M(z)-\overline{T_M(\zeta)}|
=\frac{|z-\overline\zeta|}{|cz+d|\,|c\zeta+d|}.
$$
The numerator and denominator of $K(T_M(z),T_M(\zeta))$ thus contain
the same factor $(|cz+d|\,|c\zeta+d|)^{-1}$. Cancelling this factor
proves the claimed invariance.
\end{proof}

\begin{lemma}\label{lem-square-summable-coefficients}
For every $z,\zeta\in\mathbb H$, the function
$$
w\longmapsto\ip{\rho_m(w)\psi_z}{\psi_\zeta}
$$
belongs to $\ell^2(\cW)$.
Moreover, the subspace
$$
\mathcal D=\operatorname{span}\{\psi_{s+\i}:s\in\R\}
$$
is dense in $\Ho$, and, for every $v\in\mathcal D$, the function
$c_v(w)=\ip{\rho_m(w)v}{v}$ belongs to $\ell^2(\cW)$.
\end{lemma}

\begin{proof}
For $z=x+\i y\in\mathbb H$, put
$$
S_z=\begin{pmatrix}\sqrt y&x/\sqrt y\\0&1/\sqrt y\end{pmatrix},
\qquad S_z\cdot\i=z.
$$
Set $N=S_\zeta^{-1}M_wS_z$. By Lemmas~\ref{lem-gaussian}
and~\ref{word-action}, together with Lemma~\ref{lem-mobius-invariance},
$$
\bigl|\ip{\rho_m(w)\psi_z}{\psi_\zeta}\bigr|
=K(M_w\cdot z,\zeta)^{3/2}
=K(N\cdot\i,\i)^{3/2}.
$$
If $N=\begin{pmatrix}\alpha&\beta\\\gamma&\delta\end{pmatrix}$, then
$$
K(N\cdot\i,\i)
=\frac{2}{|\beta-\gamma+\i(\alpha+\delta)|}
=\frac{2}{\sqrt{\norm{N}_{\mathrm F}^2+2}},
$$
where the last equality uses $\alpha\delta-\beta\gamma=1$.
Since $S_z$ and $S_\zeta$ are fixed invertible matrices,
$\norm{N}_{\mathrm F}\ge c_{z,\zeta}\norm{M_w}_{\mathrm F}$
for some $c_{z,\zeta}>0$. We therefore obtain
\begin{equation}\label{eq-gaussian-coefficient-decay}
\bigl|\ip{\rho_m(w)\psi_z}{\psi_\zeta}\bigr|
\le C_{z,\zeta}\norm{M_w}_{\mathrm F}^{-3/2}.
\end{equation}
The first assertion follows from \eqref{eq-matrix-summability}.
Expanding $v$ as a finite linear combination of packets then proves
$c_v\in\ell^2(\cW)$ for every $v\in\mathcal D$.

To prove density, suppose that $f\in\Ho$ is orthogonal to
$\psi_{s+\i}$ for every $s\in\R$. Oddness and the substitution $r=x^2$
give
\begin{equation}\label{eq:ort}
    0=\ip{f}{\psi_{s+\i}}
=2^{5/4}\sqrt\pi\int_0^\infty
f(\sqrt r)\e^{-\pi r}\e^{-\pi\i sr}\d r.
\end{equation}
Define $h(r)=f(\sqrt r)\e^{-\pi r}$ for $r>0$ and extend $h$ by zero
to $r\le0$. Then $h\in L^1(\R)$, since
$$
\int_0^\infty |f(\sqrt r)|\e^{-\pi r}\d r
=2\int_0^\infty |f(x)|x\e^{-\pi x^2}\d x<\infty
$$
by Cauchy--Schwarz. By \eqref{eq:ort}, we obtain
$\widehat h(s/2)=0$ for every $s\in\R$. Uniqueness of the Fourier
transform on $L^1(\R)$ gives $h=0$, hence $f=0$ on $(0,\infty)$
and, by oddness, on $\R$. Thus $\mathcal D$ is dense in $\Ho$.
\end{proof}

Finally, we finish the proof of Proposition~\ref{prp:norm-comparison}.
\begin{proof}[Proof of Proposition~\ref{prp:norm-comparison}]
Let $P\in\C[\cW]$ be a noncommutative Laurent polynomial, and set
$a=P^*P$ and $A=\rho_m(a)=\rho_m(P)^*\rho_m(P)\ge0$.
Fix a unit vector $v\in\mathcal D$.
For each integer $n\ge1$, the scalar spectral theorem for the positive
operator $A$ and Jensen's inequality give
\begin{equation}\label{eq-positive-moment-bound}
\ip{Av}{v}^{\,n}\le\ip{A^nv}{v}.
\end{equation}
Write $a^n=\sum_w b_{n,w}w$, a finite sum. Since
$\lambda(a)^n\delta_o=\sum_w b_{n,w}\delta_w$, the Cauchy--Schwarz
inequality and Lemma~\ref{lem-square-summable-coefficients} imply
\begin{align*}
0\le\ip{A^nv}{v}
&=\left|\sum_w b_{n,w}c_v(w)\right|\\
&\le\left(\sum_w|b_{n,w}|^2\right)^{1/2}\norm{c_v}_{\ell^2(\cW)}\\
&=\norm{\lambda(a)^n\delta_o}_{\ell^2(\cW)}
  \norm{c_v}_{\ell^2(\cW)}\\
&\le\norm{\lambda(a)}^n\norm{c_v}_{\ell^2(\cW)}.
\end{align*}
Combining this with \eqref{eq-positive-moment-bound}, taking $n$th
roots, and letting $n\to\infty$, we find
$$
\norm{\rho_m(P)v}_2^2=\ip{Av}{v}
\le\norm{\lambda(a)}=\norm{\lambda(P)}^2.
$$
The density of $\mathcal D$ now yields
\begin{equation}\label{eq-reverse-norm-bound}
\norm{P(U_m,V_m)|_{\Ho}}
\le\norm{P(\mathbf X,\mathbf Y)}.
\end{equation}
Together with Proposition~\ref{prop-transfer}, this proves
\begin{equation*}
    \norm{P(U_m,V_m)|_{\Ho}}=\norm{P(\mathbf{X},\mathbf{Y})}.
\end{equation*}

Consequently, the map
$$
P(\mathbf X,\mathbf Y)\longmapsto P(U_m,V_m)|_{\Ho}
$$
is a well-defined isometric unital $*$-homomorphism on the algebra of
Laurent polynomials in $\mathbf X,\mathbf Y$. It extends uniquely by
operator-norm completion to an isometric unital $*$-homomorphism\footnote{A $*$-homomorphism is a complex linear map $\pi$ between operator algebras satisfying $\pi(AB)=\pi(A)\pi(B)$ and $\pi(A^*)=\pi(A)^*$ for all $A,B$. Here $A^*$ and $\pi(A)^*$ denote the adjoints of $A$ and $\pi(A)$ respectively.}
$$
\pi_m:C_r^*(\cW)\longrightarrow\mathcal B(\Ho).
$$
Since this extension is isometric, it is injective and therefore faithful. For $\varphi,\psi\in C(\T)$, continuous functional calculus gives
$$
\pi_m\bigl(\varphi(\mathbf X)\psi(\mathbf Y)\bigr)
=
\varphi(U_m)\psi(V_m)|_{\Ho}.
$$
Since $\pi_m$ is isometric, this proves
\eqref{norm-comparison-main}.
\end{proof}

\section{A direct proof of the free-projection norm formula}
\label{app:free-projections}

In \cite[Section~3.2, p.~321]{Collins05}, the full spectral law for two free projections is obtained by a standard $S$-transform argument. The proof in Section~\ref{free-projections} applies that law and the faithfulness of the canonical trace to obtain the norm. 

Here we offer a self-contained proof of the same norm identity using alternating centered words and a Jacobi operator, without computing the full spectral density or atom weights.

\subsection{From free projections to a Jacobi operator}

The next observation will later allow us to pass from the cyclic
Jacobi model back to the full operator norm.

\begin{lemma}\label{spectral-detection}
For a reduced word $g$, define
$$
\mathbf{R}_g\delta_w=\delta_{\red(wg)}.
$$
Then $\mathbf{R}_g$ commutes with $\mathbf{X}$ and $\mathbf{Y}$. More generally, let $S$ be a bounded self-adjoint operator on $\ell^2(\cW)$ that commutes with every $\mathbf{R}_g$. Then, for every open interval $J$,
$$
\one_J(S)\Omega=0 \quad\Longrightarrow\quad \one_J(S)=0.
$$
In particular, the scalar spectral measure of $S$ associated with $\Omega$ has support equal to $\spec(S)$.
\end{lemma}

 \begin{proof}
Left and right concatenation commute before reduction and therefore after cancelling adjacent inverse pairs, so $\mathbf{R}_g$ commutes with $\mathbf{X}$ and $\mathbf{Y}$.

Now suppose that $S$ commutes with every $\mathbf{R}_g$. Since $S$ is self-adjoint,
each spectral projection $\one_J(S)$ also commutes with every $\mathbf{R}_g$. If $\one_J(S)\Omega=0$, then
$$
\one_J(S)\delta_g = \one_J(S)\mathbf{R}_g\Omega = \mathbf{R}_g\one_J(S)\Omega = 0
$$
for every reduced word $g$. Since the vectors $\delta_g$ form an orthonormal basis of
$\ell^2(\cW)$, it follows that $\one_J(S)=0$.
\end{proof}

    With notations \eqref{stBC}-\eqref{equ-pq} in mind, put
\begin{equation}\label{alpha-beta}
\alpha=\sqrt{s(1-s)}, \qquad \beta=\sqrt{t(1-t)},
\end{equation}
and define the normalized centered operators
\begin{equation}\label{EF}
\mathbf{E}=\frac{p-sI}{\alpha}, \qquad \mathbf{F}=\frac{q-tI}{\beta}.
\end{equation}

\begin{lemma}\label{lem-centered}
 The centered operators $\mathbf{E},\mathbf{F}$ satisfy
$$
\ip{\mathbf{E}\Omega}{\Omega} =\ip{\mathbf{F}\Omega}{\Omega}=0, \qquad \norm{\mathbf{E}\Omega}  =\norm{\mathbf{F}\Omega}=1.
$$
Moreover,
\begin{equation}
\mathbf{E}^2=I+\chi_s\mathbf{E}, \qquad \mathbf{F}^2=I+\chi_t\mathbf{F}, \qquad \chi_s=\frac{1-2s}{\alpha},  \quad \chi_t=\frac{1-2t}{\beta}. \label{equ-center-1}
\end{equation}
\end{lemma}

 \begin{proof}
The zero expectations follow from Lemma~\ref{lem-root}. Since $p=p^*=p^2$ and $\norm\Omega=1$,
\begin{align*}
\norm{(p-sI)\Omega}^2 &=\ip{(p-sI)^2\Omega}{\Omega}\\
&=(1-2s)\ip{p\Omega}{\Omega}+s^2\\
&=(1-2s)s+s^2=s(1-s).
\end{align*}
Thus $\norm{\mathbf{E}\Omega}=1$, and the calculation for $\mathbf{F}$ is identical. Finally, substituting $p=sI+\alpha\mathbf{E}$ into $p^2=p$ gives
$$
 \alpha^2\mathbf{E}^2 =s(1-s)I+\alpha(1-2s)\mathbf{E},
$$
and the relation for $\mathbf{F}$ follows in the same way from $q=tI+\beta\mathbf{F}$.
\end{proof}

The next lemma uses the reduced-word structure to show that
distinct alternating words in the centered operators
$\mathbf{E},\mathbf{F}$, applied to $\Omega$, give orthonormal vectors.

 \begin{lemma}\label{alternating}
 Every nonempty alternating product of $\mathbf{E}$ and $\mathbf{F}$ has zero expectation at $\Omega$:
 $$
 \ip{Z_1Z_2\cdots Z_N\Omega}{\Omega}=0, \qquad Z_j\in\{\mathbf{E},\mathbf{F}\}, \quad   Z_j\ne Z_{j+1}.
 $$
 Consequently, vectors obtained by applying distinct alternating words in $\mathbf{E},\mathbf{F}$ to $\Omega$ are orthonormal.
  \end{lemma}

\begin{proof}
By Lemma~\ref{lem-centered}, the operators
$\mathbf E\in\mathcal A_1$ and $\mathbf F\in\mathcal A_2$
have zero trace. Lemma~\ref{lem-free-independence}
therefore gives
\begin{equation}\label{alternating-0}
\ip{Z_1\cdots Z_m\Omega}{\Omega}=0
\end{equation}
for every nonempty alternating word in $\mathbf E$ and
$\mathbf F$.

  It remains to prove orthogonality. Let $W_1,W_2$ be alternating centered words. Since $\mathbf{E}$ and $\mathbf{F}$ are self-adjoint,
   \begin{equation}\label{equ-alt-2}
    \ip{W_1\Omega}{W_2\Omega} =  \ip{W_2^*W_1\Omega}{\Omega}.
  \end{equation}
 If the first letters of $W_1$ and $W_2$ are different, then $W_2^*W_1$ is a nonempty alternating word, and \eqref{alternating-0} shows that the right-hand
  side of \eqref{equ-alt-2} is zero. If the first letters agree, apply the corresponding quadratic relation in \eqref{equ-center-1} at the junction of $W_2^*$
  and $W_1$. The identity term removes the two matching letters, while the remaining term is again a nonempty alternating word and has zero expectation by \eqref{alternating-0}. Repeating this reduction gives
 \begin{equation*}
 \ip{W_1\Omega}{W_2\Omega} =
  \begin{cases}
  1,& W_1=W_2,\\
  0,& W_1\ne W_2.
  \end{cases} \qedhere
  \end{equation*}
  \end{proof}

We now use the alternating-word vectors from Lemma~\ref{alternating} to construct a Jacobi-operator representation of the compression $pqp$ on its cyclic subspace.

For the projections $p,q$ in \eqref{equ-pq} and the empty-word
vector $\Omega$, define
\begin{equation}\label{cyclic-space}
\mathcal H_{\rm cyc}
=\overline{\operatorname{span}}
\{W\Omega: W\text{ is a finite word in }p,q\}.
\end{equation}
Here $\overline{\operatorname{span}}$ denotes the closed linear span,
namely the norm closure in $\ell^2(\mathcal W)$ of all finite
linear combinations of these vectors.
Since $p^2=p$ and $q^2=q$, every word reduces to an alternating word.
Thus $\mathcal H_{\rm cyc}$ is equivalently the closed linear span of
\[
\Omega,\quad p\Omega,\quad q\Omega,\quad
pq\Omega,\quad qp\Omega,\quad
pqp\Omega,\quad qpq\Omega,\quad
pqpq\Omega,\quad qpqp\Omega,\quad\ldots.
\]

The centered alternating-word basis constructed above gives an explicit description of the compression of $pqp$ to the $p$-range of this cyclic subspace.

\begin{proposition}\label{prop-Jacobi}
Let $p,q$ be defined in \eqref{equ-pq}, and let $\mathcal{H}_{\rm cyc}$ be defined by \eqref{cyclic-space}. Set
$$
T=pqp\big|{p\mathcal{H}_{\rm cyc}}.
$$
There exists an orthonormal basis $(v_n)_{n\ge0}$ of $p\mathcal{H}_{\rm cyc}$ such that
 \begin{align}
 Tv_0&=t\,v_0+s\beta\,v_1+c\,v_2, \label{Jacobi-0}\\
Tv_1&=s\beta\,v_0+d\,v_1+c\,v_3, \label{Jacobi-1}\\
Tv_n&=c\,v_{n-2}+d\,v_n+c\,v_{n+2}, \qquad n\ge2, \label{Tn}
\end{align}
where
\begin{equation}\label{equ-Jacobi-cd}
c=\sqrt{s(1-s)t(1-t)}=\alpha\beta, \qquad d=s(1-t)+t(1-s).
\end{equation}
 Moreover,
  \begin{equation}\label{app-norm-reduction}
\norm{pq}^2=\norm{pqp}=\norm T.
\end{equation}
\end{proposition}

  \begin{proof}
 For $n\ge1$, let $w_n$ be the vector obtained by applying to $\Omega$
the alternating centered word of length $n$ beginning with $\mathbf{F}$, and set
\begin{equation}\label{wu}
w_0=\Omega, \qquad u_n=\mathbf{E} w_n, \qquad n\ge0.
\end{equation}
By Lemma~\ref{alternating}, the family
\begin{equation}\label{equ-wu-2}
\{w_n,u_n:n\ge0\}
\end{equation}
is orthonormal. Conversely, the quadratic relations \eqref{equ-center-1} reduce every word in $\mathbf{E},\mathbf{F}$ to a linear combination of alternating words. Since
$$
p=sI+\alpha\mathbf{E}, \qquad q=tI+\beta\mathbf{F},
 $$
and conversely $\mathbf{E}$ and $\mathbf{F}$ are linear combinations of $I,p$ and $I,q$, respectively, the pairs $(p,q)$ and $(\mathbf{E},\mathbf{F})$ generate the same cyclic subspace from $\Omega$. Hence \eqref{equ-wu-2} is an orthonormal basis of $\mathcal{H}_{\rm cyc}$.

By the definition \eqref{wu}, using $\mathbf{E}^2=I+\chi_s\mathbf{E}$ from \eqref{equ-center-1}, we also have
$$
\mathbf{E} u_n = \mathbf{E}^2w_n = w_n+\chi_su_n.
$$
Thus, on $\operatorname{span}\{w_n,u_n\}$,
$$
p=sI+\alpha\mathbf{E} \sim
\begin{pmatrix}
s&\alpha\\
\alpha&1-s
\end{pmatrix},
$$
 which is the rank-one projection onto
\begin{equation}\label{vn}
  v_n=\sqrt{s}\,w_n+\sqrt{1-s}\,u_n.
\end{equation}
Consequently, $(v_n)_{n\ge0}$ is an orthonormal basis of $p\mathcal{H}_{\rm cyc}$ and
\begin{equation}\label{equ-pwu}
pw_n=\sqrt{s}\,v_n, \qquad pu_n=\sqrt{1-s}\,v_n.
\end{equation}

For $n\ge2$, the relation $\mathbf{F}^2=I+\chi_t\mathbf{F}$ in \eqref{equ-center-1} gives
\begin{equation}\label{F-interior}
\mathbf{F}w_n=u_{n-2}+\chi_t w_n, \qquad \mathbf{F}u_n=w_{n+2}.
\end{equation}
Combining \eqref{vn}, \eqref{equ-pwu}, and \eqref{F-interior}, we obtain
\begin{equation}\label{pF}
p\mathbf{F}v_n = \alpha v_{n-2}+s\chi_t v_n+\alpha v_{n+2}.
\end{equation}
Since $q=tI+\beta\mathbf{F}$, equations \eqref{pF} and \eqref{equ-Jacobi-cd} yield
\begin{align*}
Tv_n &=t v_n+\beta p\mathbf{F}v_n\\
&=c v_{n-2} +\bigl(t+s(1-2t)\bigr)v_n +c v_{n+2}\\
&=c v_{n-2}+d v_n+c v_{n+2},
\end{align*}
which proves \eqref{Tn}.

At the boundary, the same relations give
\begin{equation}\label{equ-F-0}
\mathbf{F}w_0=w_1,\qquad \mathbf{F}u_0=w_2,\qquad \mathbf{F}w_1=w_0+\chi_t    w_1,\qquad \mathbf{F}u_1=w_3.
\end{equation}
Using \eqref{vn}, \eqref{equ-pwu}, and \eqref{equ-F-0},
\begin{equation}\label{pF-boundary}
p\mathbf{F}v_0=s v_1+\alpha v_2, \qquad p\mathbf{F}v_1=s   v_0+s\chi_t v_1+\alpha v_3.
\end{equation}
Therefore $q=tI+\beta\mathbf{F}$ and \eqref{equ-Jacobi-cd} give
\eqref{Jacobi-0} and \eqref{Jacobi-1}. This proves the Jacobi representation.

It remains to show that the restriction of $pqp$ to $\mathcal{H}_{\rm cyc}$ has the same norm as $pqp$. Set
$$
S=pqp.
$$
Suppose that $\norm{S|_{\mathcal H_{\rm cyc}}}<\norm S$.
Choose
\[
\norm{S|_{\mathcal H_{\rm cyc}}}<a<\norm S,
\qquad J=(a,\infty).
\]
Since $S\ge0$, $\sigma(S|_{\mathcal H_{\rm cyc}})
\subset[0,\norm{S|_{\mathcal H_{\rm cyc}}}]\subset[0,a).$
Moreover, $S$ maps $\mathcal H_{\rm cyc}$ and its orthogonal
complement into themselves. As $\Omega\in\mathcal H_{\rm cyc}$,
spectral calculus gives
\[
\one_J(S)\Omega
=\one_J(S|_{\mathcal H_{\rm cyc}})\Omega=0.
\]
On the other hand, positivity of $S$ and $a<\norm S$
imply $\one_J(S)\ne0$, contradicting
Lemma~\ref{spectral-detection}. Hence
\begin{equation}\label{equ-cyclic-3}
\norm{S|_{\mathcal H_{\rm cyc}}}=\norm S.
\end{equation}

Finally, since $S|_{\mathcal H_{\rm cyc}}=T\oplus0$
with respect to the decomposition
$\mathcal H_{\rm cyc}
=p\mathcal H_{\rm cyc}\oplus(I-p)\mathcal H_{\rm cyc}$,
\eqref{equ-cyclic-3} gives 
$$
\norm{pqp}=\norm T.
$$
Since $p$ and $q$ are orthogonal projections,
$$
  \norm{pq}^2 = \norm{(pq)(pq)^*} = \norm{pqp} = \norm T,
$$
which proves \eqref{app-norm-reduction}.
\end{proof}

\subsection{The norm formula and the half-density transition}

We now compute the norm of the Jacobi operator from Proposition~\ref{prop-Jacobi}.

\begin{lemma}\label{band-edge}
Let $T=pqp\big|{p\mathcal{H}_{\rm cyc}}$, and set
  \begin{equation}\label{app-edge-formula}
  x_\pm := d\pm2c = \left( \sqrt{s(1-t)}    \pm \sqrt{t(1-s)}  \right)^2.
\end{equation}
  Then
  $$
  \spec_{\rm ess}(T)=[x_-,x_+].
  $$
  In particular,
  $$
  x_+\in\spec_{\rm ess}(T), \qquad \norm T\ge x_+.
$$
  \end{lemma}

\begin{proof}
By Proposition~\ref{prop-Jacobi}, if
$$
\mathcal{H}_r = \overline{\operatorname{span}}\{v_{2k+r}:k\ge0\}, \qquad r=0,1,
$$
then
$$
p\mathcal{H}_{\rm cyc} = \mathcal{H}_0\oplus\mathcal{H}_1.
$$
Let $J_r$ be the constant half-line Jacobi operator on $\mathcal{H}_r$ with diagonal $d$ and off-diagonal $c$ with respect to the basis $(v_{2k+r})_{k\ge0}$; explicitly,
  \begin{align*}
J_rv_r&=d\,v_r+c\,v_{r+2},\\
    J_rv_{2k+r}  &=c\,v_{2k+r-2}+d\,v_{2k+r}+c\,v_{2k+r+2}, \qquad k\ge1.
\end{align*}

Comparing these formulas with \eqref{Jacobi-0}--\eqref{Tn}, we obtain the exact decomposition
   \begin{equation}\label{equ-T-2}
T = J_0\oplus  J_1+K,
\end{equation}
where
$$
K = (t-d)\,\ip{\cdot}{v_0}v_0  + s\beta\,\ip{\cdot}{v_1}v_0  + s\beta\,\ip{\cdot}{v_0}v_1.
$$
In particular, $\Ran K\subset\operatorname{span}\{v_0,v_1\}$, so $K$ has finite rank.

Under the unitary identifications
 $$
v_{2k}\longleftrightarrow e_k, \qquad v_{2k+1}\longleftrightarrow e_k, \qquad k\ge0,
$$
both $J_0$ and $J_1$ are unitarily equivalent to the constant half-line Jacobi operator $J$ on $\ell^2(\N_0)$ given by
  \begin{align*}
Je_0&=d\,e_0+c\,e_1,\\
Je_k&=c\,e_{k-1}+d\,e_k+c\,e_{k+1}, \qquad k\ge1.
 \end{align*}
Let $\mathcal{U}:\ell^2(\N_0)\longrightarrow L^2(0,\pi)$ be the discrete sine transform
  $$
(\mathcal{U} f)(\theta)  = \sqrt{\frac{2}{\pi}} \sum_{k=0}^{\infty} f_k\sin((k+1)\theta).
$$
 Then $\mathcal{U}$ is unitary, and the identity
$$
\sin(k\theta)+\sin((k+2)\theta) = 2\cos\theta\,\sin((k+1)\theta)
  $$
gives
$$
(\mathcal{U} Jf)(\theta) = \bigl(d+2c\cos\theta\bigr)(\mathcal{U} f)(\theta).
$$
So $\mathcal{U}J\mathcal{U}^{-1}$ is multiplication by $d+2c\cos\theta$, and thus
$$
\spec(J)=[d-2c,d+2c].
$$
Hence
$$
\spec_{\rm ess} (J_0\oplus J_1) = [d-2c,d+2c].
$$
Since $K$ is finite rank, Weyl's theorem and \eqref{equ-T-2} give
$$
\spec_{\rm  ess}(T)=[d-2c,d+2c].
$$

Finally, using \eqref{equ-Jacobi-cd},
\begin{align*}
d\pm2c &= s(1-t)+t(1-s) \pm2\sqrt{s(1-s)t(1-t)}\\
&= \left( \sqrt{s(1-t)} \pm \sqrt{t(1-s)} \right)^2,
\end{align*}
  which proves \eqref{app-edge-formula}.
\end{proof}

We next determine the possible spectrum above $x_+$.
\begin{lemma}\label{lem-point-spec}
Let $T=pqp\big|{p\mathcal{H}_{\rm cyc}}$ and $x_+$ be as in \eqref{app-edge-formula}. Then
$$
\spec(T)\cap(x_+,\infty) =
\begin{cases}
\varnothing,& s+t\le1,\\
\{1\},& s+t>1.
\end{cases}
$$
In the second case, $1$ is an eigenvalue of $T$.
    \end{lemma}
\begin{proof}
By Lemma~\ref{band-edge}, every spectral point above $x_+$ is an isolated eigenvalue of finite multiplicity. Let $\lambda$ be such an eigenvalue, with
eigenvector $x=\sum_{n\ge0}x_n v_n\ne0$. The characteristic equation associated with \eqref{Tn} has two positive reciprocal roots. Since the coefficients of an eigenvector are square summable, only the root $r\in(0,1)$ can occur. Thus we have
$$
x_{2k}=\xi_{\rm e}r^k,\qquad x_{2k+1}=\xi_{\rm o}r^k\quad(k\ge0),
$$
where
\begin{equation}\label{char-r}
\lambda=d+c(r+r^{-1}).
\end{equation}
The equations at the first two coordinates give
\begin{align}
(\lambda-t-cr)\xi_{\rm e}&=s\beta\xi_{\rm o}, \label{boundary-A}\\
(\lambda-d-cr)\xi_{\rm o}&=s\beta\xi_{\rm e}. \label{boundary-B}
\end{align}
Let $D(r)=(\lambda-t-cr)(\lambda-d-cr)-s^2\beta^2$ be the determinant of this system. By \eqref{char-r}, we have
\begin{equation}\label{det-factor}
r^2D(r)=-s^2\beta^2 \left(r-\sqrt{\frac{(1-s)(1-t)}{st}}\right) \left(r+\sqrt{\frac{(1-s)t}{s(1-t)}}\right).
\end{equation}
Since $D(r)=0$ and $r\in(0,1)$, it follows that
\begin{equation}\label{r-critical}
r=\sqrt{\frac{(1-s)(1-t)}{st}},\qquad s+t>1.
\end{equation}
For this value of $r$, we have $cr=(1-s)(1-t)$ and $c/r=st$. Substituting these identities into \eqref{char-r}, we obtain $\lambda=1$.

Conversely, assume that $s+t>1$. Let $r$ be as in \eqref{r-critical} and let $\tau=\sqrt{(1-t)/t}$. Define
\begin{equation}\label{equ-eta}
 \eta=\sum_{k=0}^{\infty}r^k(v_{2k}+\tau v_{2k+1}).
\end{equation}
Since $r\in(0,1)$, this vector is square summable. Substitution shows that it satisfies the recurrence and both boundary
equations for $\lambda=1$. Hence $T\eta=\eta$. Finally, $1=d+c(r+r^{-1})>d+2c=x_+$, which proves the assertion.
\end{proof}

\begin{proof}[Direct proof of Proposition~\ref{prop-tree}]
Recall that $T$ is the restriction of $pqp$ on $p\mathcal{H}_{\rm cyc}$,	we have $0\le T\le I.$ By \eqref{app-norm-reduction},
$$
\norm{pq}^2=\norm T.
$$

Suppose first that $s+t<1$. By Lemma~\ref{band-edge}, $x_+$ is the upper endpoint of the essential spectrum of $T$. Any spectral point of $T$ above $x_+$ is therefore an isolated eigenvalue of finite multiplicity. By Lemma~\ref{lem-point-spec}, no such eigenvalue exists. Hence
$$
\norm T=x_+.
   $$
Using \eqref{app-edge-formula}, we obtain
$$
\norm{pq}  = \sqrt{s(1-t)}+\sqrt{t(1-s)}.
 $$

If $s+t=1$, then \eqref{app-edge-formula} gives $x_+=1$. Since $x_+\in\spec_{\rm ess}(T)$ and $T\le  I$,
\begin{equation}\label{T-norm1}
\norm T=1, \qquad \norm{pq}=1.
\end{equation}

Finally, if $s+t>1$, then Lemma~\ref{lem-point-spec} gives $1\in\spec(T)$. Thus \eqref{T-norm1} also holds.

These three cases are exactly the definition of $\Phi(s,t)$.
\end{proof}

  \section*{Acknowledgements}
  M. Wang was partially supported by the National Natural Science Foundation of China (No.~12571260), the Natural Science Foundation of Hunan Province
  (No.~2026JJ20012) and the Hunan Basic Science Research Center for Mathematical Analysis (2024JC2002).

\section*{Declaration of AI use}

During the preparation of this work, the authors used OpenAI’s GPT-5.6 Sol as an interactive discussion tool and to generate the entire Lean formalization. Both uses were guided by the authors throughout the research process. The model is also used to literature searching and language polishing.

At an early stage, the model found a particular $9/10$-Wolff thin set yielding a counterexample to the SVW question. The authors guided the model to examine the structure
of this example and extended the construction to half density. Attempts to establish half density as the exact critical threshold at the original scale were unsuccessful. The authors then guided the model to investigate a more natural choice of spatial scale, leading to the asymptotic sharpness result at half density. 

All content was thoroughly rewritten, reviewed, and verified by the authors, who take full responsibility for the final manuscript.

  \end{document}